\documentclass{article}
\usepackage[pdf]{pstricks}
\usepackage[off]{auto-pst-pdf}

\usepackage[english,german]{babel}
\usepackage[T1]{fontenc}
\usepackage[ansinew]{inputenc}
\usepackage{graphicx}
\usepackage{enumerate}
\usepackage{wrapfig}
\usepackage{array}
\usepackage[automark,headsepline]{scrpage2}
\ifoot{}
\ofoot{}
\usepackage{longtable}
\usepackage{amsmath} 
\usepackage{amsthm} 

\usepackage{amssymb}

\newcolumntype{C}[1]{>{\centering}m{#1}}
\author{Tobias Lejczyk and Catharina Stroppel}

\newcommand{\R}{\mathbb{R}}

\newcommand{\Z}{\mathbb{Z}}

\newcommand{\N}{\mathbb{N}}
\newcommand{\W}{\mathcal{W}}
\newcommand{\Nmod}{\mathcal{N}}
\newcommand{\la}{\lambda}
\newcommand{\La}{\Lambda}
\newcommand{\op}{\operatorname}
\newcommand{\down}{\vee}
\newcommand{\up}{\wedge}

\renewcommand{\S}{\mathcal{S}}
\renewcommand{\L}{\mathcal{L}}
\renewcommand{\H}{\mathcal{H}}
\newcommand{\p}{\mathfrak{p}}
\renewcommand{\O}{\mathcal{O}}

\newtheorem{lem}{Lemma}[subsection]
\newtheorem{cor}[lem]{Corollary}
\newtheorem{thm}[lem]{Theorem}
\newtheorem*{thmnonr}{Theorem}
\newtheorem{prop}[lem]{Proposition}
\newtheorem{example}[lem]{Example}
\theoremstyle{definition}
\newtheorem{defi}[lem]{Definition}
\theoremstyle{remark}

\newtheorem{rem}[lem]{Remark}

\newenvironment{enum}{
\begin{enumerate}
  \setlength{\itemsep}{1pt}
  \setlength{\parskip}{0pt}
  \setlength{\parsep}{0pt}}{\end{enumerate}
}
\newenvironment{ite}{
\begin{itemize}
  \setlength{\itemsep}{1pt}
  \setlength{\parskip}{0pt}
  \setlength{\parsep}{0pt}}{\end{itemize}
}

\usepackage{graphicx,tikz}
\title{A graphical description of $(D_n,A_{n-1})$ Kazhdan-Lusztig polynomials}
\begin{document}
\psset{unit=0.5cm}		
\psset{dimen=middle}	
\newrgbcolor{lightred}{1 0.7 0.7}
\newrgbcolor{darkgreen}{0.2 0.7 0.2}
\selectlanguage{english}

\date{}
\maketitle
\setcounter{tocdepth}{1}
\begin{abstract}
We give an easy diagrammatical description of the parabolic Kazhdan-Lusztig polynomials for the Weyl group $W_n$ of type $D_n$ with parabolic subgroup of type $A_n$ and consequently an explicit counting formula for the dimension of the morphism spaces between indecomposable projective objects in the corresponding category $\O_0^\p$. As a byproduct we categorify irreducible $W_n$-modules corresponding to pairs of one-line partitions. Finally we indicate the motivation for introducing the combinatorics by connections to Springer theory, the category of perverse sheaves on isotropic Grassmannians and to Brauer algebras which will be treated in \cite{ES}.
\end{abstract}
\tableofcontents
\rohead{\thepage}
\rehead{A graphical description of $(D_n,A_{n-1})$ Kazhdan-Lusztig polynomials}
\lohead{\leftmark}
\lehead{\thepage}
\pagenumbering{arabic}

\section*{Introduction}
Finite-dimensional representations for a complex semisimple Lie algebra are well understood. Thanks to Weyl's Theorem they decompose into highest weight modules with characters given by Weyl's character formula. For infinite-dimensional highest weight modules the character theory is much harder, but Kazhdan-Lusztig polynomials, introduced by Kazhdan and Lusztig in \cite{KL79}, can be used to deduce such character formulas. Given an irreducible, possibly infinite-dimensional, highest weight module $L(\lambda)$ of highest weight $\lambda$, the Kazhdan-Lusztig polynomials, depending on $\lambda$ and $\mu$, determine the multiplicities $[M(\mu):L(\lambda)]$ of how often it occurs in a Jordan-H\"older series of a Verma module $M(\mu)$ of highest weight $\mu$. The character $\operatorname{ch}M(\mu)$ of $M(\mu)$ is known by Kostant's explicit combinatorial partition function. Using the Kazhdan-Lusztig polynomials and the formula $\operatorname{ch}M(\mu)=\sum[M(\mu):L(\lambda)]\operatorname{ch}L(\lambda)$, it is possible to express $\operatorname{ch}L(\lambda)$ as an alternating sum of characters $\operatorname{ch}M(\mu)$ and to compute the characters of the irreducible highest weight modules. To determine the inductively defined polynomials is however in general rather cumbersome.\\
In this article we will study the Lie algebra $\mathfrak{g}=\mathfrak{so}_{2n}$ of type $D_n$ and irreducible highest weight modules which are `almost' finite-dimensional; precisely, which are locally finite-dimensional for the standard parabolic subalgebra of type $A_{n}$ and have regular highest weights. These modules are objects in $\O^\p(\mathfrak{so_{2n}})_0$, the parabolic Bernstein-Gelfand-Gelfand category $\mathcal{O}$. Their characters can then be expressed as an alternating sum containing the characters of parabolic Verma modules using parabolic Kazhdan-Lusztig polynomials. Characters of parabolic Verma modules can be written in terms of characters of ordinary Verma modules using an explicit formula \cite[Prop. 3.4]{Soe97}.
In this article we will give:
\begin{ite}
 \item Explicit formulas for parabolic KL-polynomials in the case $(D_n,A_{n-1})$,
 \item a diagram calculus,
 \item explicit dimension formulas for the hom-spaces of indecomposable projective objects in $\O^\p(\mathfrak{so_{2n}})_0$,
 \item a framework for an explicit description of the endomorphism algebra of a minimal projective generator in $\O^\p(\mathfrak{so_{2n}})_0$.
\end{ite}

We start by recalling the properties of the Weyl group $\W$ of type $D_n$ with its set of generators $\S=\{s_0,\hdots,s_{n-1}\}$ and the maximal parabolic subgroup $\W_\p$ of type $A_{n-1}$ and consider the generic Hecke algebra $\H$ attached to $\W$ with its standard basis $H_w$, $w\in \W$ as $\Z[q,q^{-1}]$-module. Inside $\H$ we have the subalgebra $\H_\p$ attached to $\W_\p$. Letting $H_s$ for all $s\in\S_\p$ acting as $-q$ turns $\L$ into an $\H_\p$-module. The parabolic $\H$-Hecke module $\Nmod$, obtained by tensoring $\L$ over $\H_\p$ with $\H$,  has a standard basis $N_w=1\otimes H_w$,  $w\in\W^\p$ and also a Kazhdan-Lusztig basis $\underline{N}_w$, $w\in\W^\p$ consisting of elements that are invariant under a certain involution and can be written as linear combinations of the standard basis with leading term $1$ and otherwise coefficients in $q\Z[q]$. We are interested in the {\it Kazhdan-Lusztig polynomials} $n_{w',w}\Z[q]$ which describe the transformation between the bases, i.e. $\underline{N}_w=\sum\limits_{w'\in\W^\p}n_{w',w} N_{w'}$, see e.g. \cite{Soe97}.

  Using an embedding of $\W$ into the symmetric group $S_{2n}$ we attach to each $w\in\W^\p$ two types of objects, a {\bf weight} $w$ and a {\bf cup diagram} $C(w)$. A weight is just an $\{\up,\down\}$-sequence obtained by an easy combinatorial rule.  It can be glued on top of a cup diagram. The result is {\it oriented} if every cup is labeled by exactly one {$\down$} and one {$\up$}. The number of clockwise oriented cups is denoted $\operatorname{cl}(w'C(w))$.
  We then consider the free $\L$-module $M_D$ with weights as basis equipped with an $\H$-module structure pulled over from $\Nmod$ to $M_D$ via the obvious isomorphism $\Phi$. To each cup diagram we associate an element in $M_D$,
$$C(w)_{M_D}=\sum q^{\frac{\operatorname{cl}(w'C(w)}{2}}w',$$
where the sum runs over all weights $w'$ such that $w'C(w)$ is oriented.
\begin{thmnonr}
$\Phi(C(w)_{M_D})=\underline{N}_w,$
i.e. the image of a cup diagram $C(W)_{M_D}$ in $\Nmod$ is the Kazhdan-Lusztig basis element $\underline{N}_w$.
\end{thmnonr}
In particular, the coefficient in $C(W)_{M_D}$ of a weight $w'$ is the Kazhdan-Lusztig polynomial, i.e. $$n_{w',w}=\begin{cases}q^{\frac{\operatorname{cl}(w'C(w)}{2}} & \mbox{if $w'C(w)$ is oriented}\\ 0 & \mbox{if $w'C(w)$ is not oriented}~~.\end{cases}$$
Note that the Kazhdan-Lusztig polynomials are all monomials which is a special feature of the Hermitian symmetric case, see e.g. \cite{Boe88}. The theorem will be applied to determine the dimension of homomorphism spaces between irreducible projective objects in the parabolic category $\O^\p_0$ using a generalized TQFT construction motivated by \cite{Khovanov}, \cite{Str09}.

Our calculus allows to describe the action of $\H$ on $M_D$ neatly. This action is not faithful, a description of the quotient of $\H$ which acts faithfully on $M_D$ will be described in terms of {\it decorated tangles} introduced by Green in \cite{Gre98} as a basis of the generalized Temperley-Lieb algebra of type $D_n$. Interestingly, the quotient will depend on the parity of $n$, but will always be cellular in the sense of \cite{GL}.  It is nice to see how these decorations result from a natural folding procedure after realizing the Weyl group of type $D_n$ inside the symmetric group $S_{2n}$. Passing to decorated tangles finally allows to study projective functors in $\O^\p(\mathfrak{so_{2n}})_0$, similar to the case of $\O^\p(\mathfrak{sl_{2n}})_0$ in \cite{BS08:III}.

 Indeed, similar result for type $A$ has been proven by Brundan and the second author in \cite[5.12]{BS08:I}. The case treated there uses a slightly easier diagram combinatorics to describe the categories of perverse sheaves on Grassmannians \cite{BS08:III}, of rational representations of the Lie supergroup $\operatorname{GL}(m|n)$ \cite{BS08:IV} and non-semisimple blocks of the walled Brauer algebras \cite{BS08:V}. A TQFT construction was used to define various generalized Khovanov algebras, Morita equivalent to the endomorphism rings of projective generators in the respective categories.

The case considered here describes the combinatorics of perverse sheaves on Grassmannians for type $D$, \cite{Braden}, and the combinatorics of the Brauer algebra, \cite{Martin}, \cite{CV}. Our main goal of this paper is nothing more than to provide the setup and framework needed later in \cite{ES} for a more thorough study of the category $\O_0^\p$ and the definition of a generalized Khovanov algebra of type $D$ proving some intriguing connections with isotropic Grassmannians, Springer fibers and Brauer algebras. The decorated tangle diagrams will provide then a concise tool to describe the combinatorics and geometry of type $D$ (as well as $B$ and $C$) Springer fibers associated to $2$-row partitions. The current paper is just setting up the framework for these further reaching results. Concerning the combinatorial results, there is some overlap with some of the independently obtained results of \cite{CV}. Our forthcoming applications are heavily based on \cite{CV} and influenced by their work.\\
{\bf Acknowledgements}: The first author thanks Anton Cox and Maud De Visscher for many interesting discussions and Volodymyr Mazorchuk for helpful comments.

\section{Preliminaries}\label{prelim}
\numberwithin{lem}{section}
We first recall some standard results for the Coxeter group of type $D_n$, $n\geq 4$, see e.g. \cite{BB05} or \cite{Hum92}.
Let $\W$ be the Weyl group of type $D_n$, i.e. $\W$ is the group generated by $\S=\{s_0,s_1,\hdots,s_{n-1}\}$ with relations $(s_is_j)^{m_{ij}}=e$ where
\begin{minipage}[b]{0.68\linewidth}
\begin{eqnarray*}
\hspace{-3pt}
m_{ij}=\begin{cases} 1 &\mbox{if} i=j,\\ 3 & \mbox{if $i$ connects $j$
in the Dynkin diagram $\Gamma$:}\\ 2 & \mbox{otherwise.} \end{cases}
\end{eqnarray*}
\vspace{-8pt}
\end{minipage}
\begin{minipage}[b]{0.3\linewidth}
\begin{pspicture}[shift=-2.2](-0.3,-2.5)(5.5,2.5)
\psline(0,1)(1,0)(2.4,0)(3,0)
\psline(0,-1)(1,0)
\psline(4.6,0)(5.2,0)
\pscircle[fillcolor=white,fillstyle=solid](0,1){2pt}
\pscircle[fillcolor=white,fillstyle=solid](0,-1){2pt}
\pscircle[fillcolor=white,fillstyle=solid](1,0){2pt}
\pscircle[fillcolor=white,fillstyle=solid](2.4,0){2pt}
\pscircle[fillcolor=white,fillstyle=solid](5.2,0){2pt}
\rput(0,-1.5){\footnotesize$0$}
\rput(0,1.5){\footnotesize$1$}
\rput(1,0.5){\footnotesize$2$}
\rput(2.4,0.5){\footnotesize$3$}
\rput(5.2,0.5){\footnotesize$n-1$}
\rput(3.8,0){$\cdots$}
\end{pspicture}
\end{minipage}

Let $l:\W\rightarrow \N$ be the length function. We fix the parabolic subgroup $\W_\p$ generated by $\S_\p=\{s_1,\hdots,s_{n-1}\}$ and denote the set of shortest representatives for $\W_\p\backslash\W$ by $\W^\p$. Hence $w\in\W^\p$ iff $l(sw)>l(w)$ for all $s_i\in\S_\p$. Write $\Z_2$ for the multiplicative group $\{+1,-1\}$. Then $\W$ is isomorphic to $G=S_n\ltimes\Z_2^n/I$, where $I=\{(\alpha_1,\ldots,\alpha_n)\mid\prod{\alpha_i}=1\}$ with the group law  $(\sigma,\alpha)\circ(\tau,\beta)=(\sigma\circ\tau, \tau(\alpha)\cdot\beta)$, where $\tau$ acts on $\alpha$ by permuting the components. This definition makes sense for any $n\geq0$. The isomorphism sends $\W_\p$ to the subgroup $S_n\ltimes \{e\}$. We fix the induced bijection
$$\W_\p\backslash \W \cong (S_n\ltimes \{e\})\backslash(S_n\ltimes \Z_2^n/I)=\Z_2^n/I$$ and identify the sets with the set $S(n)$ of $\{+,-\}$-sequences of length $n$ with an even number of minuses, e.g. for $n=4$:
\begin{eqnarray}
\label{halfseq}
\stackrel{|++++]}{e}\;\;\stackrel{|--++]}{s_0}\;\;\stackrel{|-+-+]}{s_0s_2}\;\;\stackrel{|-++-]}{s_0s_2s_3}\;\;
\stackrel{|+--+]}{s_0s_2s_1}\;\;\stackrel{|+-+-]}{s_0s_2s_3s_1}\;\;\stackrel{|++--]}{s_0s_2s_3s_1s_2}\;\;
\stackrel{|----]}{s_0s_2s_3s_1s_2s_0}.
\end{eqnarray}

The group $\W$ embeds into $S_{2n}$, viewed as the permutation group of the set $\{-n,\hdots,-1,1,\hdots,n\}$, via
\begin{eqnarray}
		\Psi:\quad \W \hookrightarrow  S_{2n} &&
s_i\mapsto\begin{cases}
t_0t_{1}t_{-1}t_0 &\text{if $i=1$,}\\		
t_{i}t_{-i} & \text{otherwise,}	
\end{cases}
		\end{eqnarray}
with the `simple transpositions' $t_0=(-1,1)$ and $t_i=(i,i+1)$ if $1\leq i\leq n-1$, and $t_i=(-i-1,-i)$ if $-(n-1)\leq i\leq -1$. The image of $\Psi$ is contained in the (type $B_n$) subgroup of signed permutations, i.e. permutations with $\sigma(-i)=-\sigma(i)$; the subgroup $\W_\p$ is sent to the subgroup $S_n\times S_n\subset S_{2n}$. Hence $\W$ acts naturally on  $\{+,-\}$-sequences of length $2n$ identifying $\W^\p$ with the set $S_{sym}(n)$ of  {\it even antisymmetric sequences}, e.g. the sign at place $i$ is the opposite of the sign at place $-i$, and there is an even number of minuses in the second half.
We write these sequences as $[\alpha_{-n},\hdots,\alpha_{-1}|\alpha_1,\hdots,\alpha_n]$ where the $|$ indicates the antisymmetry, so that the $|\alpha_1,\hdots,\alpha_n]$ are nothing else than our sequences as in \eqref{halfseq}. The $\{+,-\}$-sequence of length $2n$ associated to $w\in\W^\p$ is denoted by $\alpha_w$.

\begin{rem}
\label{Youngdiag}
The set $S_{sym}(n)$ is in canonical bijection with $Y_{sym}(n)$, the set of Young diagrams that are symmetric with respect to the main diagonal, have an even number of boxes on the main diagonal and fit into an $n\times n$ square. Given a $\{+,-\}$-sequence the Young diagram is constructed as follows: Start at the upper right corner of an $n\times n$-square. Reading the sequence from right to left go one step down for a minus and one step left for a plus. Reflect this path on the main diagonal and fill the space to the left of this path with boxes, e.g.
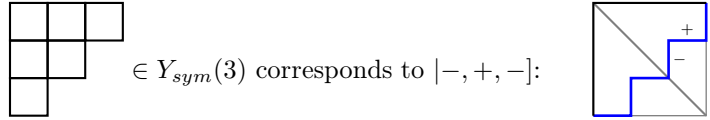
\begin{center}
	\begin{pspicture}[shift=-1.5](0,-0.5)(3,3.5)
	\psframe(0,0)(1,1)
	\psframe(0,1)(1,2)
	\psframe(1,1)(2,2)
	\psframe(0,2)(1,3)
	\psframe(1,2)(2,3)
	\psframe(2,2)(3,3)
	\end{pspicture}
	$\in Y_{sym}(3)$ corresponds to $|-,+,-]$:\quad\quad
	\begin{pspicture}[shift=-1.5](0,-0.5)(3,3.5)
	\psframe[linecolor=gray, linewidth=0.5pt](0,0)(3,3)
	\psline[linecolor=gray](0,3)(3,0)
	\psline[linecolor=blue, linewidth=1pt](3,3)(3,2)(2,2)(2,1)(1,1)(1,0)(0,0)
	\psline(0,0)(0,3)(3,3)
	\rput(2.3,1.5){\tiny$-$}
	\rput(2.5,2.3){\tiny$+$}
    	\end{pspicture}
\end{center}
The generator $s_i$ for $1\leq i\leq n-1$ corresponds to adding/deleting a box on the $i$th and $-i$th diagonal, and in the case $i=0$ to adding/deleting four boxes arranged in a $2\times 2$-array on the main diagonal.
 A reduced expression can be read off the Young diagram by writing the numbers of the diagonals into the boxes, working with boxes of size $2\times 2$ on the main diagonal, e.g.

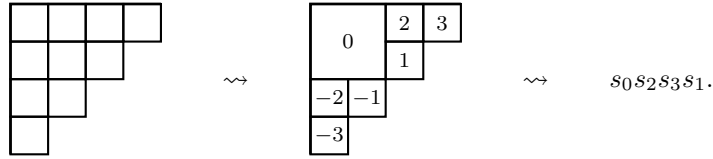
\begin{center}
\begin{pspicture}[shift=-2.3](0,-0.5)(4,4.5)	
\psline[linewidth=1pt](0,0)(0,4)(4,4)
	\psframe(0,2)(1,3) \psframe(0,3)(1,4) \psframe(1,2)(2,3) \psframe(1,3)(2,4) 
	\psframe(0,1)(1,2) \psframe(2,3)(3,4)
	\psframe(1,1)(2,2) \psframe(2,2)(3,3)
	\psframe(0,0)(1,1) \psframe(3,3)(4,4)
	\end{pspicture}
$\quad\quad\rightsquigarrow\quad\quad$
	 \begin{pspicture}[shift=-2.3](0,-0.5)(4,4.5)\psline[linewidth=1pt](0,0)(0,4)(4,4)
	\psframe(0,2)(2,4) 
	\psframe(0,1)(1,2) \psframe(2,3)(3,4)
	\psframe(1,1)(2,2) \psframe(2,2)(3,3)
	\psframe(0,0)(1,1) \psframe(3,3)(4,4)
	\rput(1,3){\footnotesize $0$}
	\rput(0.5,1.5){\footnotesize $-2$}
	\rput(2.5,3.5){\footnotesize $2$}
	\rput(1.5,1.5){\footnotesize $-1$}
	\rput(2.5,2.5){\footnotesize $1$}
	\rput(0.5,0.5){\footnotesize $-3$}
	\rput(3.5,3.5){\footnotesize $3$}
	\end{pspicture}
$\quad\quad\rightsquigarrow \quad\quad s_0s_2s_3s_1$.\\
\end{center}
The resulting reduced expressions for $n=4$ are given by \eqref{halfseq}.
\end{rem}

We finally recall a few facts from \cite{KL79}, \cite{Soe97}. Let $\H$ be the Hecke algebra for  $\W$ over the ring $\L=\Z[q,q^{-1}]$ of formal Laurent polynomials. Hence it is free as $\L$-module on basis $H_w$, $w\in \W$, subject to the algebra relations $H_{w'}H_w=H_{w'w}$ if $l(w')+l(w)=l(w'w)$ and $H_s^2=(q^{-1}-q)H_s+H_e$ for $s\in \S$. A set of generators for $\H$ is $\{C_s=H_s+q\mid s\in\S\}$; the Hecke algebra $\H_\p$ for $W_\p$ is the subalgebra generated by the $C_s$, $s\in\S_\p$. Letting $C_s$ act on $\L$ by zero for all $s\in\S_p$, turns $\L$ into a right $\H_p$-module. Define $\Nmod=\Nmod^\p=\L\otimes_{\H_\p}\H$ with its  {\it standard basis} $\{N_w\mid w\in\W^\p\}$ where $N_w=1\otimes H_w$. The action of $C_s$ on $\Nmod$ is given by
\begin{eqnarray}
\label{N}
N_wC_s&=&\begin{cases}N_{ws}+qN_w & \mbox{if } ws\in W^p \mbox{ and } l(ws)>l(w)\\
N_{ws}+q^{-1}N_w & \mbox{if } ws\in W^p \mbox{ and } l(ws)<l(w)\\
0 & \mbox{if }ws\notin W^p.
\end{cases}
\end{eqnarray}
The $\mathbb{Z}$-linear involution $H_w\mapsto H_{w^{-1}}^{-1}$, $q\mapsto q^{-1}$ on $\H$ induces an involution on $\Nmod$. The {\it Kazhdan-Lusztig basis}, denoted $\{\underline{N}_w\mid w\in\W^\p\}$, is characterized by the fact that each basis element is fixed under this involution and satisfies $\underline{N}_w\in N_w +\sum_{w'\not=w} q\Z[q] N_{w'}$. The {\it Kazhdan-Lusztig polynomials} $n_{w,w'}$ are the entries in the base change matrix between the two bases, i.e. $$\underline{N}_w=\sum\limits_{w'} n_{w',w} N_{w'}.$$
Starting with $\underline{N}_e=N_e$ one can compute the $\underline{N}_w$'s inductively: assuming $\underline{N}_z$ is known for $l(z)<l(w)$, first choose a simple reflection $s_i$ such that $l(ws_i)<l(w)$, and write $\underline{N}_{ws_i}C_{s_i}=\sum_{z\in W^\p} m_z N_z$ in the standard basis. Then
\begin{eqnarray}
\label{Ns}
\underline{N}_w&=&\underline{N}_{ws_i}C_{s_i}-\sum_{z\in W^\p,\; l(z)<l(w)} m_z(0) \underline{N}_z.
\end{eqnarray}

\section{A graphical description of KL-polynomials}
In this section we will prove the main theorem giving an explicit graphical formula for the Kazhdan-Lusztig polynomials in our parabolic case.

\subsection{Extended graphical weights and cup diagrams}\label{chapter_cup_diagrams}
\begin{defi}\label{def_weight}
	The \textit{(extended) diagrammatical weight} $w_{\down\up}$ associated to an element $w\in\W^\p$ is obtained from $\alpha_w$ by first extending the full $\{+,-\}$-sequence of length $2n$ with $n$ pluses to the left and $n$ minuses to the right and then replacing every {$+$} with a {$\down$} and every {$-$} with an {$\up$}. We call the $n$ pluses and $n$ minuses by which we extended the sequence \textit{frozen}.
\end{defi}

\begin{defi}
Let $M_D$ be the free $\L$-module with basis  $w_{\down\up}$, $w\in\W^\p$ and $\H$-module structure induced from $\Nmod$ via the canonical isomorphism $N_w\mapsto w_{\down\up}$.
\end{defi}

To define a cup diagram associated to a $\{+,-\}$-sequence we first define a matching and modify this matching afterwards to get our cup diagram.
	\begin{defi} Consider the set of integers $P:=\{-2n,\hdots,-1,1,\hdots,2n\}\subset \Z\subset \R^2$ as points on the number line. A \textit{matching} is a diagram, up to isotopy, consisting of $2n$ non-intersecting arcs in the lower half plane $\R\times\R_{\leq 0}$ connecting the $4n$ points without intersections and such that any point in $P$ is connected with precisely one other point in $P$ via an arc.
	\end{defi}
	\begin{defi} Given a $\{+,-\}$-sequence $\alpha$ the \textit{corresponding matching} $M(\alpha)$ is defined as follows:
	First extend $\alpha$ by $n$ pluses to the left and $n$ minuses to the right. Then construct a matching where each arc connects a {$+$} with a {$-$} and the point of the arc labeled  {$+$} is smaller than the one labeled {$-$}.
	\end{defi}
		\begin{lem}
	Given a $\{+,-\}$-sequence $\alpha$ then $M(\alpha)$ exists and is well-defined.
	\begin{proof}
	Note that a pair $p$ of a {$+$} directly to the left of a {$-$} has to be connected by an arc. Since otherwise the {$-$} would be connected to a {$+$} further to the left making it impossible to connect the original {$+$} to a {$-$} with a higher index without intersecting the arc.
Ignoring the pair $p$ (and the arc) gives a shorter $\{+,-\}$-sequence and we are done by induction (noting that there is always a pair $p$ for extended diagrammatical weights).
	\end{proof}
	\end{lem}

		\begin{defi}
		The \textit{cup diagram} $C(w)$ associated with $w\in\W^\p$ is obtained from the matching $M(\alpha_w)$ by `pairing' the arcs crossing the middle. Starting from the point $(0,0)$ we take the first two arcs crossing the middle, exchange two bounary points, and put a decoration on the created intersection indicating that they are linked (see e.g. \eqref{a2}).
We continue this process with the next two arcs until no unlinked arcs crossing the middle are left. The resulting arcs are also called {\it cups}. We denote the set of cup diagrams attached to $\W^\p$ by $C(\W^\p)$.
		\end{defi}
The following shows that $C(w)$ is well-defined:
\begin{lem} For all $w\in\W^\p$ the number of arcs crossing the middle in $M(\alpha_w)$ is even. The assignment $w \mapsto C(w)$ defines a bijection
	$\W^\p \rightarrow  C(\W^\p).$
\begin{proof}
The number of arcs crossing the middle equals $2n-2\sum_{\alpha_i=+1, i>0}1$.
(It simply counts the number of pluses in the right half of the sequence noting that every plus there is connected to a minus in the same half, whereas the remaining minuses have to be connected to the left half.) To two different elements in $\W^\p$ we associate two different $\{+,-\}$-sequences. To these sequences we obviously associate two different matchings and as a consequence two different cup diagrams. So the map is injective. It is surjective by definition.
\end{proof}
\end{lem}
	
In Definition \ref{def_weight}, to every element in $\W^\p$ we associated a weight. Because of the extension we can glue the weight associated to a $w'\in\W^\p$ on top of a cup diagram $C(w)$. This gives us a new diagram, denoted $w'C(w)$.
	\begin{defi}
		A cup is called \textit{oriented} if it is labeled by exactly one {$\up$} and one {$\down$}. This goes for linked cups, too. They are treated as if they were single cups. The decoration has no meaning for the orientation. A cup diagram $w'C(w)$ is an \textit{oriented cup diagram} if each cup is oriented. We denote by $\operatorname{cl}(w'C(w))$ the number of clockwise oriented cups in an oriented cup diagram $w'C(w)$, i.e. oriented cups of the form~~$\begin{pspicture}[shift=-0.2](0,0)(1,-0.5)\psellipticarc(0.5,0)(0.5,0.5){180}{0}\rput(0,0){\footnotesize\boldmath$\up$} \rput(1,0.2){\footnotesize\boldmath$\down$}\end{pspicture}$~.
	\end{defi}
	\begin{rem}
		For linked cups the antisymmetry of the weights implies that both cups are oriented in the same direction.
	\end{rem}
For a cup diagram $C(w)$ define
\begin{eqnarray}
C(w)_{M_D}&=&\sum v^{\frac{\operatorname{cl}(w'C(w))}{2}}w'_{\down\up},
\end{eqnarray}
where the sum is over all $w'_{\down\up}\in\W^\p$ such that $w'C(w)$ is oriented.\\

Next we analyze the difference between the cup diagram $C(w)$ and the cup diagram $C(ws_i)$ in the case that $l(ws_i)>l(w)$ and $w,ws_i\in\W^\p$. If multiplying $w$ by $s_i$ yields a longer element in $\W^\p$, this means for the $\{+,-\}$-sequence one of the following two cases:
\begin{eqnarray*}
|\hdots-+\hdots]\leadsto|\hdots+-\hdots],&&
|++\hdots]\leadsto|--\hdots].
\end{eqnarray*}
(To translate this into cup diagrams, consider the reflected antisymmetric part as well and add the pluses resp. minuses needed to connect all the given pluses and minuses according to our rules for cup diagrams.)
In the first case we have to connect the minus in the first sequence to a plus to the left. It matters whether this plus is in the first or second half of the sequence, leading to the cases \eqref{a1} and \eqref{a2}:
\begin{eqnarray}
\label{a1}
\begin{pspicture}[shift=-0.5](-4,0)(4,-0.5)
\psline[linecolor=gray, linewidth=0.5pt](0,1)(0,-1.2)
\psellipticarc(1,0)(0.5,0.5){180}{0}
\psellipticarc(3,0)(0.5,0.5){180}{0}
\rput(1.5,0.3){\tiny$i$}
\rput(2.5,0.3){\tiny$i+1$}
\psellipticarc(-1,0)(0.5,0.5){180}{0}
\psellipticarc(-3,0)(0.5,0.5){180}{0}
\rput(-1.5,0.3){\tiny$-i$}
\rput(-3.5,0.7){\tiny$+$}
\rput(-2.5,0.7){\tiny$-$}
\rput(-1.5,0.7){\tiny$+$}
\rput(-0.5,0.7){\tiny$-$}
\rput(0.5,0.7){\tiny$+$}
\rput(1.5,0.7){\tiny$-$}
\rput(2.5,0.7){\tiny$+$}
\rput(3.5,0.7){\tiny$-$}
\end{pspicture}
&\leadsto&
\begin{pspicture}[shift=-1.5](-4.5,0)(3,-1.5)
\psline[linecolor=gray, linewidth=0.5pt](-0.5,1)(-0.5,-1.2)
\psellipticarc(1.5,0)(0.5,0.5){180}{0}
\psellipticarc(1.5,0)(1.5,1.2){180}{0}
\rput(1,0.3){\tiny $i$}
\rput(2,0.3){\tiny $i+1$}
\psellipticarc(-2.5,0)(0.5,0.5){180}{0}
\psellipticarc(-2.5,0)(1.5,1.2){180}{0}
\rput(-2,0.3){\tiny $-i$}
\rput(-4,0.7){\tiny$+$}
\rput(-3,0.7){\tiny$+$}
\rput(-2,0.7){\tiny$-$}
\rput(-1,0.7){\tiny$-$}
\rput(0,0.7){\tiny$+$}
\rput(1,0.7){\tiny$+$}
\rput(2,0.7){\tiny$-$}
\rput(3,0.7){\tiny$-$}
\end{pspicture}\\
\label{a2}
\begin{pspicture}[shift=-1.5](-4.5,0)(4.5,-1.5)
\psline[linecolor=gray, linewidth=0.5pt](0,0)(0,-1.7)
\psellipticarc(2.5,0)(0.5,0.5){180}{0}
\psellipticarc(1.5,0)(2.5,1.5){180}{0}
\rput(1,0.3){\tiny$i$}
\rput(2,0.3){\tiny$i+1$}
\psellipticarc(-2.5,0)(0.5,0.5){180}{0}
\psellipticarc(-1.5,0)(2.5,1.5){180}{0}
\rput(-1,0.3){\tiny$-i$}
\pscircle*(0,-1.2){2pt}
\rput(-4,0.7){\tiny$+$}
\rput(-3,0.7){\tiny$+$}
\rput(-2,0.7){\tiny$-$}
\rput(-1,0.7){\tiny$+$}
\rput(1,0.7){\tiny$-$}
\rput(2,0.7){\tiny$+$}
\rput(3,0.7){\tiny$-$}
\rput(4,0.7){\tiny$-$}
\end{pspicture}
&\leadsto&
\begin{pspicture}[shift=-1.5](-4.5,0)(4,-1.5)
\psline[linecolor=gray, linewidth=0.5pt](0,0)(0,-1.7)
\psellipticarc(1.5,0)(0.5,0.5){180}{0}
\psellipticarc(0.5,0)(3.5,1.5){180}{0}
\rput(1,0.3){\tiny$i$}
\rput(2,0.3){\tiny$i+1$}
\psellipticarc(-1.5,0)(0.5,0.5){180}{0}
\psellipticarc(-0.5,0)(3.5,1.5){180}{0}
\rput(-1,0.3){\tiny$-i$}
\pscircle*(0,-1.47){2pt}
\rput(-4,0.7){\tiny$+$}
\rput(-3,0.7){\tiny$+$}
\rput(-2,0.7){\tiny$+$}
\rput(-1,0.7){\tiny$-$}
\rput(1,0.7){\tiny$+$}
\rput(2,0.7){\tiny$-$}
\rput(3,0.7){\tiny$-$}
\rput(4,0.7){\tiny$-$}
\end{pspicture}
\end{eqnarray}
In the second case we just complete the cups and get
\begin{eqnarray}
\label{b1}
\begin{pspicture}[shift=-0.5](-4.5,0)(3.5,-0.5)
\psline[linecolor=gray, linewidth=0.5pt](-0.5,0)(-0.5,-1.2)
\psellipticarc(-2.5,0)(0.5,0.5){180}{0}
\psellipticarc(-2.5,0)(1.5,1.2){180}{0}
\psellipticarc(1.5,0)(0.5,0.5){180}{0}
\psellipticarc(1.5,0)(1.5,1.2){180}{0}
\rput(0,0.3){\tiny$1$}
\rput(1,0.3){\tiny$2$}
\rput(0,0.7){\tiny$+$}
\rput(1,0.7){\tiny$+$}
\end{pspicture}
\leadsto
\begin{pspicture}[shift=-2.5](-4.5,0)(2,-2.5)
\psline[linecolor=gray, linewidth=0.5pt](0,0)(0,-2.2)
\psellipticarc(0.5,0)(1.5,1){180}{0}
\psellipticarc(-0.5,0)(1.5,1){180}{0}
\psellipticarc(0.5,0)(3.5,2){180}{0}
\psellipticarc(-0.5,0)(3.5,2){180}{0}
\pscircle*(0,-0.95){2pt}
\pscircle*(0,-1.97){2pt}
\rput(1,0.3){\tiny $1$}
\rput(2,0.3){\tiny $2$}
\rput(1,0.7){\tiny $-$}
\rput(2,0.7){\tiny $-$}
\end{pspicture}
\end{eqnarray}
%
%
%
%
\subsection{The Kazhdan-Lusztig basis described as cup diagrams}
The following provides our graphical description of the Kazhdan-Lusztig basis.
\begin{thm}[Graphical Kazhdan-Lusztig basis]\label{main_theorem} The image of a cup diagram under the isomorphism $\Phi$ is the corresponding Kazhdan-Lusztig basis element.
$$\Phi(C(w)_{M_D})=\underline{N}_w.$$
and the Kazhdan-Lusztig polynomial $n_{v,w}$ equals
\begin{eqnarray*}
n_{v,w}&=&\begin{cases} q^\frac{\operatorname{cl}(vC(w))}{2} & \mbox{if } vC(w)\mbox{ is oriented}\\ 0 & \mbox{if } vC(w)\mbox{ is not oriented .} \end{cases}
\end{eqnarray*}
\end{thm}
\begin{proof}
We prove this by induction on $l(w)$, starting with $\underline{N}_e=N_e$, i.e. $n_{v,e}=\delta_{v,e}$ for $\nu\in\W^\p$. We identified $e=[\underbrace{-\hdots-}_{n}|\underbrace{+\hdots+}_{n}]$, so we get the cup diagram

\begin{eqnarray}
\begin{pspicture}(-6,0)(6,-1.7)
\psline[linecolor=gray, linewidth=0.5pt](0,0)(0,-1.7)
\psellipticarc(4,0)(0.5,0.5){180}{0}
\psellipticarc(4,0)(1.5,0.9){180}{0}
\psellipticarc(4,0)(3.5,1.5){180}{0}
\psellipticarc(-4,0)(0.5,0.5){180}{0}
\psellipticarc(-4,0)(1.5,0.9){180}{0}
\psellipticarc(-4,0)(3.5,1.5){180}{0}
\rput(1.5,0){$\hdots$} \rput(6.5,0){$\hdots$}
\rput(-1.5,0){$\hdots$} \rput(-6.5,0){$\hdots$}
\rput(0.5,0.2){\tiny $1$} \rput(3.5,0.2){\tiny $n$} \rput(7.5,0.2){\tiny $2n$}
\rput(-0.5,0.2){\tiny $-1$} \rput(-3.5,0.2){\tiny $-n$} \rput(-7.5,0.2){\tiny $-2n$}
\end{pspicture}
\end{eqnarray}
Now we determine the possible weights. First, we have to take the ``frozen'' orientations into account. Between $-2n$ and $-(n+1)$ we get $n$ {$\down$}'s and between $n+1$ and $2n$ we get $n$ {$\up$}'s. So the only possibility to get an oriented diagram is the weight $[\underbrace{\down\hdots\down}_{n\mbox{\tiny~~frozen}}\underbrace{\up\hdots\up}_{n}|\underbrace{\down\hdots\down}_{n}\underbrace{\up\hdots\up}_{n\mbox{\tiny~~frozen}}]$ which corresponds to $e$. There are no clockwise oriented cups, so our formula gives exactly the polynomials $n_{v,e}$, and hence the basis for our induction.\\
Assume the formula holds for all $w'\in \W^\p$, $l(w')<l(w)$. Choose $s_i$ such that $l(ws_i)<l(w)$ and recall \eqref{Ns}.
We first compare the effect of multiplication with $C_{s_i}$ on the algebraic with the diagrammatic side on $\underline{N}_{ws_{i}}$.
Assume $N_y$ appears in $\underline{N}_{ws_{i}}$ with some nontrivial coefficient (which by induction is a power of $q$).
From \eqref{N}, we have three cases:
\begin{enumerate}
\item $l(ys_{i})>l(y)$ and $ys_{i}\in\W^\p$:\\
In this case $N_{y}C_{s_i}=N_{ys_i}+vN_{y}$. We recall that for a weight getting longer means exchanging an {$\up$} at $i$ with a {$\down$} at places $i+1$ for $1\leq i\leq n-1$ resp. exchanging two {$\down$}'s at places $1$ and $2$ into two {$\up$}'s if $i=0$. We work through the three cases \eqref{a1}, \eqref{a2}, \eqref{b1} describing the passage from $C(w)$ to $C(ws_i)$.
		\begin{ite}
		\item
				For~~\begin{pspicture}(-3.5,0.5)(3.5,-0.5)
				\psline[linecolor=gray, linewidth=0.5pt](0,0.5)(0,-0.7)
				\psellipticarc(1,0)(0.5,0.5){180}{0}
				\psellipticarc(3,0)(0.5,0.5){180}{0}
				\rput(1.5,0.6){\tiny$i$}
				\rput(2.5,0.6){\tiny$i+1$}
				\psellipticarc(-1,0)(0.5,0.5){180}{0}
				\psellipticarc(-3,0)(0.5,0.5){180}{0}
				\rput(-1.5,0.3){\tiny$-i$}
				\rput(1.48,0){\footnotesize\boldmath$\up$}
				\rput(2.5,0.2){\footnotesize\boldmath$\down$}
				\end{pspicture}~~to be oriented we need to have the orientation~~ \begin{pspicture}(-3.5,0.6)(3.5,-0.5)
				\psline[linecolor=gray, linewidth=0.5pt](0,0.5)(0,-0.7)
				\psellipticarc(1,0)(0.5,0.5){180}{0}
				\psellipticarc(3,0)(0.5,0.5){180}{0}
				\rput(1.5,0.6){\tiny$i$}
				\rput(2.5,0.6){\tiny$i+1$}
				\psellipticarc(-1,0)(0.5,0.5){180}{0}
				\psellipticarc(-3,0)(0.5,0.5){180}{0}
				\rput(-3.5,0.2){\footnotesize\boldmath$\down$}
				\rput(-2.52,0){\footnotesize\boldmath$\up$}
				\rput(-1.5,0.2){\footnotesize\boldmath$\down$}
				\rput(-0.52,0){\footnotesize\boldmath$\up$}
				\rput(0.5,0.2){\footnotesize\boldmath$\down$}
				\rput(1.48,0){\footnotesize\boldmath$\up$}
				\rput(2.5,0.2){\footnotesize\boldmath$\down$}
				\rput(3.48,0){\footnotesize\boldmath$\up$}
				\end{pspicture}~.
None of these cups are oriented clockwise. The part of the weight assigned to $y$ that is relevant for our calculation is $\down\up\down\up|\down\up\down\up$ (involving only the above four cups). The corresponding part of $ys_i$ is then $\down\down\up\up|\down\down\up\up$. Putting them on top of $C(ws_i)$ gives
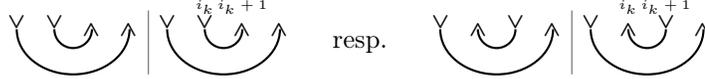
\begin{eqnarray*}
				\begin{pspicture}[shift=-1](-4,1)(3,-1.5)
				\psline[linecolor=gray, linewidth=0.5pt](-0.5,0.5)(-0.5,-1.2)
				\psellipticarc(1.5,0)(0.5,0.5){180}{0}
				\psellipticarc(1.5,0)(1.5,1.2){180}{0}
				\psellipticarc(-2.5,0)(0.5,0.5){180}{0}
				\psellipticarc(-2.5,0)(1.5,1.2){180}{0}
				\rput(1,0.6){\tiny$i_k$}
				\rput(2,0.6){\tiny$i_k+1$}
				\rput(-4,0.2){\footnotesize\boldmath$\down$}
				\rput(-3,0.2){\footnotesize\boldmath$\down$}
				\rput(-2.02,0){\footnotesize\boldmath$\up$}
				\rput(-1.03,0){\footnotesize\boldmath$\up$}
				\rput(0,0.2){\footnotesize\boldmath$\down$}
				\rput(1,0.2){\footnotesize\boldmath$\down$}
				\rput(1.98,0){\footnotesize\boldmath$\up$}
				\rput(2.97,0){\footnotesize\boldmath$\up$}
				\end{pspicture}
&\quad \text{resp.} \quad&
				\begin{pspicture}[shift=-1](-4,1)(3,-1.5)
				\psline[linecolor=gray, linewidth=0.5pt](-0.5,0.5)(-0.5,-1.2)
				\psellipticarc(1.5,0)(0.5,0.5){180}{0}
				\psellipticarc(1.5,0)(1.5,1.2){180}{0}
				\psellipticarc(-2.5,0)(0.5,0.5){180}{0}
				\psellipticarc(-2.5,0)(1.5,1.2){180}{0}
				\rput(1,0.6){\tiny$i_k$}
				\rput(2,0.6){\tiny$i_k+1$}
				\rput(-4,0.2){\footnotesize\boldmath$\down$}
				\rput(-3,0){\footnotesize\boldmath$\up$}
				\rput(-2,0.2){\footnotesize\boldmath$\down$}
				\rput(-1.03,0){\footnotesize\boldmath$\up$}
				\rput(0,0.2){\footnotesize\boldmath$\down$}
				\rput(1,0){\footnotesize\boldmath$\up$}
				\rput(2,0.2){\footnotesize\boldmath$\down$}
				\rput(2.97,0){\footnotesize\boldmath$\up$}
				\end{pspicture}
\end{eqnarray*}
Either of them is oriented, the first has no clockwise oriented cup; whereas the second has two additional clockwise oriented cups as expected. The cases \eqref{a2} and \eqref{b1} are treated similarly.
		\end{ite}
		
\item $l(ys_{i})<l(y)$ and $ys_{i}\in\W^\p$:\\
		In this case we have the multiplication rule $N_{y}C_{s_{i}}=N_{ys_{i}}+q^{-1}N_{y}$.
We recall that getting shorter for a weight means exchanging an {$\up$} from $i+1$ with a {$\down$} at $i$ for $1\leq i\leq n-1$ resp. exchanging two {$\up$}'s at $1$ and $2$ to two {$\down$}'s if $i=0$. Now the considerations are analogous to the ones before. Again one has to work through the three cases and count the change in clockwise oriented cups and obtains $\operatorname{cl}(ys_{i}C(w))=\operatorname{cl}(yC(ws_{i}))$ and $\operatorname{cl}(yC(w))=\operatorname{cl}(yC(ws_{i}))-2$ as expected.
	\item $ys_{i}\notin\W^\p$:\\
		This time the multiplication rule is $N_{y}C_{s_{i}}=0$. If $y\in\W^\p$ but $ys_{i}\notin\W^\p$, then the weight associated with $y$ does not change when applying $s_i$. For $1\leq i\leq n-1$ this means we have either two {$\down$}'s or two {$\up$}'s at $i$ and $i+1$; and for $i=0$ we either have a {$\down$} at $1$ and an {$\up$} at $2$ or vice versa. The same procedure as above indeed yields only diagrams which are not oriented.
\end{enumerate}
Hence our calculations show that diagrammatically we get all the contributions predicted algebraically. Moreover, we do not get any multiplicities, since one verifies easily on the diagram side using \eqref{a1}, \eqref{a2}, \eqref{b1} that if $N_y$ occurs in $\underline{N}_{ws_i}$ then $N_{xs_i}$ does not occur in   $\underline{N}_{ws_i}$ for any $y\in \W^\p$.
The only question left is whether there are any weights providing an oriented diagram besides those from the calculations above. But looking at cases \eqref{a1} and \eqref{a2} we see that all possible orientations of the new diagram occur, hence no additional unexpected weights occur. Therefore the diagrammatical formulas determine $\underline{N}_{ws_i}C_{s_i}$\\
Finally, we may have to subtract possibly some $\underline{N}_z$, $z\not=w$ as in \eqref{Ns},
i.e. those terms arising from a constant term in the coefficient. In terms of diagrams this would mean that we have an oriented cup diagram with no clockwise oriented cups. But there is exactly one weight leading to such an oriented cup diagram, namely the weight corresponding to $w$. Hence there does not occur any other polynomials with constant term and we get  $\underline{N}_w=\underline{N}_{ws_{i}}C_{s_{i}}$. The coefficients in $C(w)_{M_D}$ are the Kazhdan-Lusztig polynomials and the image under $\Phi$ of a cup diagram is the Kazhdan-Lusztig basis element.
\end{proof}
\begin{rem}
The proof of Theorem~\ref{main_theorem} shows that for all $w\in\W^\p$ with $w=s_{i_1}\cdots s_{i_k}$ an arbitrarily chosen reduced expression we have $\underline{N}_w=\underline{N}_eC_{s_{i_1}}\cdots C_{s_{i_k}}$. Elements for which the Kazhdan-Lusztig basis element has this property are called \textit{Deodhar} and were first studied in \cite{Deo90}. Jointly with Warrington and Jones respectively, Billey classified Deodhar elements first for type $A$ in \cite{BW01} and then for general type in \cite{BJ07}. The above Kazhdan-Lusztig polynomials were already studied in detail by Martin, \cite{Martin}, and Cox and de Visscher, \cite{CV}, in the context of Brauer algebras, and combinatorially by Brenti \cite{Bre09} using skew shifted s-Dyck partitions, see \cite{Lej10} for a concrete dictionary.
\end{rem}

\section{Homomorphism spaces of projectives in $\O^\p_0$}
Let  $\O^\p_0=\O^\p_0(\mathfrak{so}_{2n})$ be the principal block of the parabolic category $\O$ for $\mathfrak{so}_{2n}$ with respect to our parabolic $\p$, (see \cite{Hum08} for details). The simple objects are precisely the simple highest weight modules with highest weights $w\cdot0$, $w\in W^\p$. Our Kazhdan-Lusztig polynomials $n_{v,w}$ evaluated at $1$ count the multiplicity of simple highest weight modules occurring as subquotients in a composition series of a parabolic Verma module. This goes back to a conjecture by Kazhdan and Lusztig in \cite[Conj. 1.5]{KL79} which was later generalized to the parabolic case and proven there by Casian and Collingwood in \cite{CC87}. Together with the BGG-reciprocity formulas we obtain the dimension of the homomorphism spaces of projective objects, namely for $w, w'\in\W^\p$
\begin{equation}
\dim\operatorname{Hom}_{\O^\p_0}(P(w'),P(w))=\sum_{v\in\W^\p}n_{v,w'}(1)n_{v,w}(1)\label{eq:dim_hom_space}
\end{equation}
where the $P(x)$, $x\in\W^\p$ are a complete set of representatives for the isomorphism classes of indecomposable projective objects.\\
%
%
%
Using Theorem \ref{main_theorem} we see that $n_{v,w}(1)$ equals $1$ if $vC(w)$ is oriented and $0$ otherwise; hence the product $n_{v,w'}(1)n_{v,w}(1)$ is $1$ if the $v$ orients both cup diagrams, $C(w)$ and $C(w')$, and is $0$ if at least one of the cup diagrams is not oriented.
\subsection{Colored circle diagrams and hom spaces}
In this section we introduce colored circle diagrams which give a convenient way to compute the number of weights orienting two cup diagrams simultaneously.
\begin{defi}
The \textit{cap diagram} associated to $w\in\W^\p$ is defined to be the vertical reflection of $C(w)$ and denoted by $\overline{C(w)}$. Gluing a weight $v$ below the cap diagram gives us $\overline{C(w)}v$.
\end{defi}
\begin{rem}
All the notions we had for oriented cup diagrams can be applied to cap diagrams, too. $\overline{C(w)}v$ is oriented if and only if $vC(w)$ is oriented and the number of clockwise oriented caps coincides with the one of clockwise oriented cups. As pictured below, this is obvious:
\begin{eqnarray*}
\begin{pspicture}(0,-0.5)(1.5,0)
\psellipticarc(0.5,0)(0.5,0.5){180}{0}
\rput(0,0){\footnotesize\boldmath$\up$}
\rput(1,0.2){\footnotesize\boldmath$\down$}
\end{pspicture}
\leadsto
\begin{pspicture}(-0.5,-0.3)(4,0)
\psellipticarc(0.5,0)(0.5,0.5){0}{180}
\rput(0,-0.2){\footnotesize\boldmath$\up$}
\rput(1,0){\footnotesize\boldmath$\down$}
\end{pspicture}
&&
\begin{pspicture}[shift=-0.5](0,-1)(1.5,0)
\psellipticarc(0.5,0)(0.5,0.5){180}{0}
\rput(1,0){\footnotesize\boldmath$\up$}
\rput(0,0.2){\footnotesize\boldmath$\down$}
\end{pspicture}
\leadsto
\begin{pspicture}(-0.5,-0.3)(1,0)
\psellipticarc(0.5,0)(0.5,0.5){0}{180}
\rput(1,-0.2){\footnotesize\boldmath$\up$}
\rput(0,0){\footnotesize\boldmath$\down$}
\end{pspicture}
\end{eqnarray*}
\end{rem}
	\begin{defi}
	The \textit{circle diagram} associated to two cup diagrams $C(w)$ and $C(w')$ for $w,w'\in\W^\p$ is obtained by gluing the cap diagram $\overline{C(w')}$ on top of the cup diagram $C(w)$. This circle diagram is denoted $\overline{C(w')}C(w)$.\\
	Writing a weight $v$ between the cup and the cap diagram gives us the diagram $\overline{C(w')}vC(w)$. We call this an \textit{oriented circle diagram} if all circles are oriented.
	\end{defi}
Obviously, a circle diagram is oriented by a weight if and only if both cup diagrams are oriented simultaneously by this weight. All possible antisymmetric orientations of such a circle diagram with the correct frozen labels, give us in fact a weight in $S_{sym}(n)$ by Lemma \ref{lem:occuring_weights} below (so that we do not have to check if a constructed weight actually is a valid weight). Hence

\begin{prop}
\label{circlediags}
The vector space $E^\p=\operatorname{End}_{\O^\p_0}(\bigoplus_{w\in\W^\p} P(w))$ underlying the endomorphism algebra of a minimal projective generator of $\O^\p_0$ can be identified with the span of all oriented circle diagrams  $\overline{C(w)}vC(w')$ for $v,w,w'\in W^\p$.
\end{prop}

		\begin{lem}\label{lem:occuring_weights}
		All antisymmetric weights with $n$ {$\up$}'s from $n+1$ to $2n$ occurring as an orientation of some $C(w)$ are in $S_{sym}(n)$, i.e. have an even number of {$\up$}'s between $1$ and $n$.
\end{lem}		

\begin{proof}
		Consider the first part of the upper half of $w$, i.e. the points $1$ through $n$.
		All pluses not connected to a minus between $1$ and $n$ are connected to a minus between $n+1$ and $2n$. Assume a plus of the second kind occurs at place $i$. The orientation of a weight at a point between $n+1$ and $2n$ is {$\up$}. Consequently, to get an oriented cup diagram the weight has to have a {$\down$} at $i$. So these places do not contribute to the number of {$\up$}'s between $1$ and $n$ and can be ignored.\\
		The number of minuses occuring in the first part of the upper half of $w$ is even. Hence there is an even number of cups ending between $1$ and $n$. If two of these cups are linked because they both come from the lower half of $w$, then antisymmetry forces them to be oriented in the same direction. Hence they do not change the parity of the number of minuses and can be ignored, too. In addition ignoring them does not change the parity of the number of cups ending between $1$ and $n$.\\
		Now we have to consider two cases: Either all the other cups are connected within the first part of the upper half of $w$ or one cup is linked to one going to a place between $n+1$ and $2n$.\\
		In the first case to get an oriented cup diagram every cup gets exactly one {$\up$} and one {$\down$}. Since the number of cups ending between $1$ and $n$ is even, the number of cups connected within this interval has to be even. So the number of {$\up$}'s is even and the first case is done.\\
		For the second case we first make some observations. The way the linked cups are nested implies that at most one cup crossing the middle is linked to one going to a place between $n+1$ and $2n$. Suppose there is exactly one cup of this kind called $c$. Then all the other relevant cups are connected within the first part of the upper half of $w$. The number of these is odd since the only other relevant cup is $c$ and the total number of relevant cups is even.\\
		The cup $c$ and its linked cup look like this:
		$$
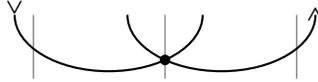
\begin{pspicture}[shift=-1](-4,0)(4,-1.5)
		\psline[linecolor=gray, linewidth=0.5pt](0,0)(0,-1.7)
		\psline[linecolor=gray, linewidth=0.5pt](-3.5,0)(-3.5,-1.7)
		\psline[linecolor=gray, linewidth=0.5pt](3.5,0)(3.5,-1.7)
		\psellipticarc(1.5,0)(2.5,1.5){180}{0}
		\psellipticarc(-1.5,0)(2.5,1.5){180}{0}
		\rput(4,0){\footnotesize\boldmath$\up$}
		\rput(-4,0.2){\footnotesize\boldmath$\down$}
		\pscircle*(0,-1.2){2pt}
		\end{pspicture}$$
		Obviously, to get an oriented cup diagram, we have to have the orientation
		$$
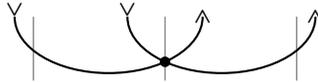
\begin{pspicture}[shift=-1](-4,0)(4,-1.5)
		\psline[linecolor=gray, linewidth=0.5pt](0,0)(0,-1.7)
		\psline[linecolor=gray, linewidth=0.5pt](-3.5,0)(-3.5,-1.7)
		\psline[linecolor=gray, linewidth=0.5pt](3.5,0)(3.5,-1.7)
		\psellipticarc(1.5,0)(2.5,1.5){180}{0}
		\psellipticarc(-1.5,0)(2.5,1.5){180}{0}
		\rput(4,0){\footnotesize\boldmath$\up$}
		\rput(-4,0.2){\footnotesize\boldmath$\down$}
		\rput(1,0){\footnotesize\boldmath$\up$}
		\rput(-1,0.2){\footnotesize\boldmath$\down$}
		\pscircle*(0,-1.2){2pt}
		\end{pspicture}$$
		So we get an {$\up$} for the cup $c$. The other cups are again oriented with exactly one plus and one minus. So the total number of {$\up$}'s is again even.
		\end{proof}

 A priori, every circle can be oriented in two different directions. The resulting weight is by Lemma~\ref{lem:occuring_weights} in $S_{sym}(n)$ if it is antisymmetric and has the correct frozen labels. These properties however restrict the possibilities for orienting circles. Some circles may only have one possible orientation and some circles even none. To distinguish these cases we color the circles.
	\begin{defi}
	We call the points bigger than $n$ \textit{upper outer points} and the points smaller than $-n$ \textit{lower outer points}. Then each circle in a circle diagram $\overline{C(x)}C(w)$ is colored according to the following rules:
	\begin{ite}
		\item {\bf Black}, if a circle does not pass through any outer points and the number of different linked pairs it contains is even.
		\item \textcolor{red}{\bf Red}, if a circle passes through more than one upper outer point, more than one lower outer point or the number of different linked pairs it contains is odd.
		\item \textcolor{green}{\bf Green}, if a circle is neither black nor red.
	\end{ite}
	Let $\operatorname{bk}(w,x)$ and $\operatorname{rd}(w,x)$ be the number of black and red circles respectively.
	\end{defi}
Obviously, the three coloring cases are disjoint and each circle gets colored. The following lemma clarifies further the conditions of the coloring rules.
	\begin{lem}\label{lem:crossings}
		\begin{enum}
		\item If a circle crosses the middle, it either always contains both cups or caps of a linked pair or it contains always only one of the linked cups and caps.
		\item If a circle always contains only one of the linked cups or caps, then the number of different linked pairs that are part of the circle is even.
		\item If a circle always contains both cups or caps, that are linked then the number of different linked pairs that are traversed is odd.
		\end{enum}
		\begin{proof}
		\begin{enum}
		\item  Assume a circle traverses only one of the cups or caps in a linked pair $A$  and contains both linked cups or caps in another linked pair $B$. This means that the traversed cup or cap in $A$ has to be connected to one cup or cap in $B$. Then, because of the symmetry of the diagram the other linked cup in $B$ has to be connected to the cup or cap in $A$ which is not traversed. But this means that the not traversed cup or cap in $A$ is connected to the same circle; a contradiction.
		\item For every circle, the number of times it crosses the middle is even. If it crosses the middle from left to right it has to cross it back to get a circle and vice versa. But this implies that if a circle contains only single cups or caps of linked pairs, the number of different pairs involved has to be even.		
		\item Consider two pairs of linked cups and caps.
		$$
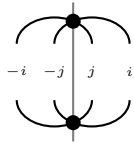
\begin{pspicture}(-2,-1.1)(2,2.3)
		\psline[linecolor=gray](0,-1.1)(0,2.6)
		\psellipticarc(0.5,0)(1,0.7){180}{0} \psellipticarc(-0.5,0)(1,0.7){180}{0} \pscircle*(0,-0.6){0.2}
		\psellipticarc(0.5,1.5)(1,0.7){0}{180} \psellipticarc(-0.5,1.5)(1,0.7){0}{180} \pscircle*(0,2.1){0.2}
		\rput(-1.5,0.75){\tiny$-i$}
		\rput(-0.5,0.75){\tiny$-j$}
		\rput(0.5,0.75){\tiny$j$}
		\rput(1.5,0.75){\tiny$i$}
		\end{pspicture}$$
		If we start connecting them without crossing the middle we could either connect $-j$ with $-j$ and because of the symmetry $j$ with $j$ or we could connect $-j$ of the bottom cup with $-i$ of the cap, and proceed analogously in the positive part.
		$$
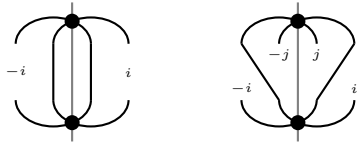
\begin{pspicture}(-2,-0.7)(4,2.3)
		\psline[linecolor=gray](0,-1.1)(0,2.6)
		\psellipticarc(0.5,0)(1,0.7){180}{0} \psellipticarc(-0.5,0)(1,0.7){180}{0} \pscircle*(0,-0.6){0.2}
		\psellipticarc(0.5,1.5)(1,0.7){0}{180} \psellipticarc(-0.5,1.5)(1,0.7){0}{180} \pscircle*(0,2.1){0.2}
		\psline(-0.5,0)(-0.5,1.5) \psline(0.5,0)(0.5,1.5)
		\rput(-1.5,0.75){\tiny$-i$}
		\rput(1.5,0.75){\tiny$i$}
		\end{pspicture}
		\begin{pspicture}(-2,-0.7)(2,2.3)
		\psline[linecolor=gray](0,-1.1)(0,2.6)
		\psellipticarc(0.5,0)(1,0.7){180}{0} \psellipticarc(-0.5,0)(1,0.7){180}{0} \pscircle*(0,-0.6){0.2}
		\psellipticarc(0.5,1.5)(1,0.7){0}{180} \psellipticarc(-0.5,1.5)(1,0.7){0}{180} \pscircle*(0,2.1){0.2}
		\psline(-0.5,0)(-1.5,1.5) \psline(0.5,0)(1.5,1.5)
		\rput(-1.5,0.3){\tiny$-i$}
		\rput(-0.5,1.2){\tiny$-j$}
		\rput(0.5,1.2){\tiny$j$}
		\rput(1.5,0.3){\tiny$i$}
		\end{pspicture}$$
		But concerning our circle and crossings of the middle the first picture would be the same as having $-i$ connected directly with $-i$ on the left half of the picture without crossing the middle. The second picture would be the same as having $-j$ from the top connected with $-i$ from the bottom omitting the loops. In both cases the parts in the positive half have to be connected symmetrically.
		$$
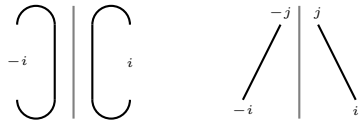
\begin{pspicture}(-2,-0.3)(4,2.3)
		\psline[linecolor=gray](0,-0.5)(0,2.5)
		\psline(-0.5,0)(-0.5,2) \psline(0.5,0)(0.5,2)
		\rput(-1.5,1){\tiny$-i$}
		\rput(1.5,1){\tiny$i$}
		\psellipticarc(-1,0)(0.5,0.5){180}{0} \psellipticarc(-1,2)(0.5,0.5){0}{180}
		\psellipticarc(1,0)(0.5,0.5){180}{0} \psellipticarc(1,2)(0.5,0.5){0}{180}
		\end{pspicture}
		\begin{pspicture}(-2,-0.3)(4,2.3)
		\psline[linecolor=gray](0,-0.5)(0,2.5)
		\psline(-1.5,0)(-0.5,2) \psline(1.5,0)(0.5,2)
		\rput(-1.5,-0.3){\tiny$-i$}
		\rput(-0.5,2.3){\tiny$-j$}
		\rput(1.5,-0.3){\tiny$i$}
		\rput(0.5,2.3){\tiny$j$}
		\end{pspicture}$$
		Hence, in terms of circles and parity of crossings of the middle, an even number of crossings can be transformed into connections between two points on each side of the middle. The pictures also show that connecting the two points without crossing any other cup or cap is possible.
		Connecting one of these pictures into one circle gives us two lines which cross the middle. These had to be paired. Hence the total number of different pairs of linked cups and caps has to be odd if the circle passes through both cups or caps that are linked.
		\end{enum}
		\end{proof}
	\end{lem}
	\begin{cor}\label{cor:selfintersections}
	All circles with self-intersections are colored red.
	\begin{proof}
	Since intersections only occur in linked pairs that cross the middle, a self-intersection meant that we have a linked pair in which the circle contains both cups or caps. But then Lemma \ref{lem:crossings}.1 says that the circle always passes through both parts of the linked pairs it contains. The number of different linked pairs traversed by the circle is then odd by Lemma \ref{lem:crossings}. This means that the circle has to be colored red.
	\end{proof}
	\end{cor}

		\begin{lem}
		Red circles can not be oriented. Black circles can be oriented in two directions. Green circles can be oriented in exactly one direction.
		\begin{proof}
		We prove this case by case.
		\begin{ite}
			\item {\bf Red}: If a circle passes through more than one upper outer point, this means two upper outer points have to be connected in some way. But the orientation at both points is {$\up$}. This prohibits orientation of the circle. A picture of this is:
						$$
						
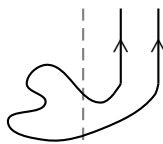
\begin{pspicture}(-2,-1.5)(2,2)
						 \psline[linecolor=gray,linestyle=dashed](0,-1.5)(0,2)
						\psline(1,0)(1,2) \psline(2,0)(2,2)
						\rput(1,1){\footnotesize\boldmath$\up$} \rput(2,1){\footnotesize\boldmath$\up$}
						 \pscurve(1,0)(0.5,-0.5)(-1,0.5)(-1.5,0)(-1,-0.5)(-2,-1)(-1,-1.5)(0.5,-1.2)(1.9,-0.3)(2,0)
						\end{pspicture}
						$$
						where the squiggly line may be any connection between the two upper outer points.\\
						The same argument works analogously for two lower outer points.\\
						We know that for any oriented circle the number of times it crosses the middle from left to right is the same as the number of crossings from right to left. If the number of different linked pairs is odd, then this implies that the circle always has to traverse both cups or caps that are linked. Otherwise the number of crossings is odd which contradicts that is an oriented circle. We also know that linked cups are always oriented in the same direction. But this implies that the numbers of crossings from left to right resp. from right to left can not be the same.
			\item {\bf Black}: By Corollary \ref{cor:selfintersections}, a black circle has no self-intersections. For all linked pairs it contains, it only passes through one part of the pair. Also, it does not pass through any outer points. These two things together imply, first of all, that the circle can be oriented and second, that the orientation can be chosen freely since no weight is fixed by any precondition.
			\item {\bf Green}: Green circles pass through at least one outer point. Otherwise, the number of different linked pairs the circle contains was either even which would lead to a black coloring or odd which would lead to a red coloring.\\
			The orientation of the weight at this outer point is fixed. Hence such a circle can be oriented in at most one direction. If it passes through no other outer point then the orientation of the weight at all other points can be chosen freely since the circle only passes through single cups or caps of linked pairs because of Lemma \ref{lem:crossings}. Hence it is possible to orient this circle in the given direction.\\
			If the circle passes through one outer point at each side, we have a picture like this
			$$
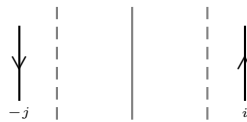
\begin{pspicture}(-3,-0.5)(3,2.5)
			\psline[linecolor=gray,linestyle=dashed](-2,-0.5)(-2,2.5) \psline[linecolor=gray,linestyle=dashed](2,-0.5)(2,2.5)
			\psline[linecolor=gray](0,-0.5)(0,2.5)
			\psline(-3,0)(-3,2) \psline(3,0)(3,2)
			\rput(3,1){\footnotesize\boldmath$\up$} \rput(-3,1){\footnotesize\boldmath$\down$}	
			\rput(3,-0.3){\tiny$i$}
			\rput(-3,-0.3){\tiny$-j$}
			\end{pspicture}
			$$
			with $i\neq j$. If $i$ was equal to $j$, then, because of the symmetry of the diagram, for all cups and caps traversed by the circle the reflected counterpart has to be traversed also. But there has to be one cup crossing the middle. Its reflected counterpart is the other part of the linked pair. Hence we get a self-intersection which, according to Corollary ~\ref{cor:selfintersections}, leads to a red coloring and not to a green one.\\
			If we wanted to connect the upper end of the left line with the lower end of the right line, to get an unoriented circle, we had to connect the upper end of the left line with the lower end of the right line. But this leads to a self-intersection of the circle, which again led to a red coloring instead of a green one. So we can rule out this case, too, and orient our circle in the given way.
		\end{ite}\vskip -2em
		\end{proof}
		\end{lem}
	\begin{rem}
	It can be shown that if a circle contains a cup and its mirrored cup then the circle has to be colored red. (The antisymmetry of the $\{+,-\}$-sequences and the way the starting points of cups crossing the middle are exchanged ensure that the mirrored cup exists and is different from the original cup. If now a circle contains a cup and its mirrored cup than it has to cross the middle somewhere. So the first cup is connected to a cup crossing the middle. Because of the reflection symmetry the mirrored cup is connected to the cup linked to the one crossing the middle. But this is a self-intersection which by Corollary~\ref{cor:selfintersections} leads to a red coloring.) In particular, reflecting a black circle $C$ always yields a circle different from $C$, since the circle can not pass through any of its mirrored cups.
	\end{rem}
Although a black circle can be oriented in both directions this does not mean that we can orient all circles independently. Because of the antisymmetry of the weights any black oriented circle determines the orientation of its reflected counterpart. So only half the black circles can be oriented without any limitations.\\
We now express \eqref{eq:dim_hom_space} in terms of counting colored circles (setting $0^0+1$):
\begin{thm}\label{thm:dim_hom}
The dimension of $\op{Hom}_{\O^\p_0}(P(w),P(w'))$ is $2^\frac{\operatorname{bk}(w,w')}{2}\cdot 0^{\operatorname{rd}(w,w')}$.
\end{thm}

\begin{proof}
We know that the dimension of $\op{Hom}_{\O^\p_0}(P(w'),P(w))$ is the number of weights $v$ such that $\overline{C(w')}vC(w)$ is oriented, i.e. all circles are oriented. If a red circle appears this means by the previous lemma that this circle can not be oriented by any weight. Hence the dimension has to be $0$.\\
If no circle is colored red this means the diagram can be oriented since green and black circles can be oriented. The only open question is how many weights orient the diagram. By the previous discussion half of the black circles can be oriented freely in both directions while green circles can only be oriented in one direction. This gives us $2^{\frac{\operatorname{bk}(w,x)}{2}}$ possible weights.\\
Orienting only half of the black circles takes care of the antisymmetry of the weight and orienting the green circles only counterclockwise ensures the orientation of the weight at the points bigger than $n$ and smaller than $-n$. Thus by Lemma \ref{lem:occuring_weights} each orientation gives us a weight in $S_{sym}(n)$.
\end{proof}

\begin{rem}
Of course green circles contribute a trivial factor $1^{\operatorname{gr}(w,x)}$, where $\operatorname{gr}(w,x)$ is the number of green circles.
\end{rem}
The formula of Theorem \ref{thm:dim_hom} resembles \cite[5.4]{Str09}. There, a `colored' version of a $2$-dim TQFT was introduced, and further developed in \cite[Section 10]{Gisa},  to describe the multiplicative structure of the endomorphism algebra. These results easily carry over to our case, see \cite{ES} for the explicit algebra structure.

%
%
%
\section{Decorated tangles $\H$ on $\Nmod$}
%
%
%
Our goal is a diagrammatical description of the $\H$-action on $\Nmod$ in terms of decorated tangles arising from generalized Temperley-Lieb algebra of type $D_n$.
\subsection{Temperley-Lieb algebras and decorated tangles}\label{gen_TL}
We start with some definitions on Temperley-Lieb algebras following \cite{Gre98}.
\begin{defi}
The {\it generalized Temperley-Lieb algebra} $\operatorname{TL}(\W)$ is the quotient of the Hecke algebra $\H(\W)$ by the ideal $I(\W)$ generated by elements of the form $\sum_{w\in\left<s_i,s_j\right>} q^{-l(w)}H_w$, where $i$ and $j$ are connected in the Dynkin diagram $\Gamma$.
\end{defi}
%
%
%
	\begin{defi}
	Let $m,n\in\N_0$ such that $m+n$ is even. Consider $m+n$ points contained in a rectangle such that $m$ points are on the bottom face of the rectangle and $n$ points on the top. An \textit{$(m,n)$-tangle} is a collection of lines and circles contained in the interior of the rectangle, such that the lines connect all the points and no intersections occur. Two such tangles are called the same if they are isotopy equivalent.\\
	Lines connecting two points at the top face are called \textit{cups} and lines connecting two points at the bottom face are called \textit{caps}. Lines connecting a point at the top with one at the bottom are called \textit{edges}. A \textit{decorated tangle} is a tangle with some `$\bullet$'s (called {\it dots}) on the lines and/or circles such that every dot is accessible from the left side of the rectangle, meaning there exists a line in the rectangle connecting the left face with {'$\bullet$'} not intersecting the tangle, see \eqref{eq:es} for examples.
	\end{defi}
\begin{rem}
We use the term `tangle' for tangles without crossings. Usually tangles are allowed to have crossings.
\end{rem}
We can concatenate an $(m,n)$-tangle with an $(n,r)$-tangle vertically which will result in an $(m,r)$-tangle. We do this by writing the $(n,r)$-tangle on top of the $(m,n)$-tangle. The $n$ points at the bottom of the $(n,r)$-tangle are then connected one to one from left to right to the $n$ points at the top of the $(m,n)$-tangle.
\begin{rem}
	Often, the decorated tangles are introduced as morphisms of a category. The objects are the natural numbers. Morphisms from $m$ to $n$ are just $(m,n)$-tangles and composition of morphisms is given by the vertical concatenation of the tangles.
\end{rem}
 For any commutative ring $R$ let $\mathbb{T}_n$ be the free $R$-algebra with basis consisting of all $(n,n)$-tangles and multiplication given by the concatenation of tangles.\\
For $1\leq i\leq n-1$ define the tangle $e_i$ to be the $(n,n)$-tangle connecting the points $i$ and $i+1$ at the north resp. south face and the point $k\not=i,i+1$ on the north face with the point $k$ at the south face. Define the tangle $e_0$ to be the tangle $e_1$ with two decorations, one on the cup connecting point $1$ with point $2$ at the top face and one on the cap connecting $1$ and $2$ at the bottom face:
\begin{equation}
\label{eq:es}
\begin{tikzpicture}[thick]
\draw (0,-.5) node {$e_2=$};
\draw (.5,0) to +(0,-1);
\draw (1,0) .. controls +(0,-.5) and +(0,-.5) .. +(.5,0);
\draw (1,-1) .. controls +(0,.5) and +(0,.5) .. +(.5,0);
\draw (2,0) to +(0,-1);
\draw (2.5,0) to +(0,-1);
\draw (3,0) to +(0,-1);

\begin{scope}[xshift=5cm]
\draw (0,-.5) node {$e_0=$};
\draw (.5,0) .. controls +(0,-.5) and +(0,-.5) .. +(.5,0);
\fill (.75,-.35) circle(2.5pt);
\draw (.5,-1) .. controls +(0,.5) and +(0,.5) .. +(.5,0);
\fill (.75,-.65) circle(2.5pt);
\draw (1.5,0) to +(0,-1);
\draw (2,0) to +(0,-1);
\draw (2.5,0) to +(0,-1);
\draw (3,0) to +(0,-1);
\end{scope}
\end{tikzpicture}
\end{equation}
Let  $\mathbb{DT}_n$ be the subalgebra of $\mathbb{T}_n$ generated by $e_0,\hdots,e_{n-1}$.

\begin{prop}[\rm{\cite{Gre98}}]
\label{prop:Green}
If $R=\L$ then $C_{s_i}\mapsto e_i$ defines an isomorphism of $R$-algebras from $\operatorname{TL}(\W)$ to $\mathbb{DT}_n$ modulo the local relations
\begin{equation}
\label{fig:TL-Rel}
\begin{tikzpicture}[thick,scale=1]
\draw circle(10pt) +(1.6,0) node{$=q+q^{-1}$};
\draw (3.3,0) +(.7,-.35) to +(.7,.45);
\fill (4,-.05) circle(2.5pt) +(0,.2) circle(2.5pt);
\draw (4.4,0) node {$=$} +(.3,-.35) to +(.3,.45);
\draw (5.4,0) +(1.05,0) circle(10pt) +(1.8,0);
\fill (6.1,0) circle(2.5pt);
\fill (7,0) circle(2.5pt);
\draw (6.3,0) +(.7,-.35) to +(.7,.45);
\draw (7.4,0) node{$=$};
\draw (7,0) +(1.05,0) circle(10pt) +(1.8,0);
\fill (7.7,0) circle(2.5pt);
\draw (7.9,0) +(.7,-.35) to +(.7,.45);
\end{tikzpicture}
\end{equation}

\end{prop}

\begin{rem} The lines involved in the relations \eqref{fig:TL-Rel} are not necessarily edges in the tangle but might also be part of a circle. The first relation means we can remove an undotted circle and instead multiply with $(q+q^{-1})$. The second relation allows us to delete an even number of decoration from any line or circle, and then the last relations tells us that in presence of a circle with just one decoration, all other decorations in the tangle may be deleted.
\end{rem}
%
%
%
\subsection{From cup diagrams to decorated tangles}
All the information necessary to build a cup diagram is located at the points $1$ to $n$. So we `cut out' this portion of the diagram. This cutting (or maybe better `folding') operation will give us a decorated tangle. More precisely,
take a cup diagram, draw vertical lines between $-1$ and $1$ and between $n$ and $n+1$ and `cut' the cups along the two lines with these rules:
\begin{enum}
\item Cups connecting two points between $1$ and $n$ stay the way they are. \label{rel:full_cup}
\item Cups from a point between $1$ and $n$ to one between $n+1$ and $2n$ turn into an edge connecting the point between $1$ and $n$ to the bottom face. \label{rel:single_line}
\item Two linked cups both ending at points between $1$ and $n$ are replaced by a dotted cup, i.e. a cup with a  `$\bullet$', connecting the two endpoints. \label{rel:dec_cup}
\item Two linked cups of which only one ends at a point between $1$ and $n$ are replaced by a dotted edge from the point between $1$ and $n$ to the `bottom face'. \label{rel:dec_line}
\item Two linked cups ending both at a number bigger than $n$ are removed. \label{rel:rem}
\end{enum}
It is easy to see that to each cup connected to a point between $1$ and $n$ there is exactly one rule that applies. Denote the tangle associated to $C(w)$ by $\underline{w}$. In fact, $\underline{w}$ can be directly constructed from $w$ by first connecting all neighbored $+-$ pairs successively by a cup ignoring already joint symbols. Then connect from the remaining $-$'s pairwise, starting from the left, by a dotted cup. Finally attach vertical edge to the remaining symbols with an additional dot in case the label is $-$. Note there is at most one dotted edge.

\begin{example}
\label{dottedcups}
Cutting the elements from $C(\W^\p)$ with $W^\p$ as in \eqref{halfseq} yields to
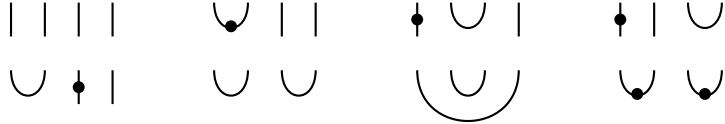
\begin{eqnarray*}
\begin{tikzpicture}[thick,scale=0.9]
\draw (.5,0) to +(0,-.5);
\draw (1,0) to +(0,-.5);
\draw (1.5,0) to +(0,-.5);
\draw (2,0) to +(0,-.5);
\begin{scope}[xshift=3cm]
\draw (0.5,0) .. controls +(0,-.5) and +(0,-.5) .. +(.5,0);
\fill (.75,-.35) circle(2.5pt);
\draw (1.5,0) to +(0,-.5);
\draw (2,0) to +(0,-.5);
\end{scope}
\begin{scope}[xshift=6cm]
\draw (.5,0) to +(0,-.5);
\fill (.5,-.25) circle(2.5pt);
\draw (1,0) .. controls +(0,-.5) and +(0,-.5) .. +(.5,0);
\draw (2,0) to +(0,-.5);
\end{scope}
\begin{scope}[xshift=9cm]
\draw (.5,0) to +(0,-.5);
\draw (1,0) to +(0,-.5);
\fill (.5,-.25) circle(2.5pt);
\draw (1.5,0) .. controls +(0,-.5) and +(0,-.5) .. +(.5,0);
\end{scope}
\begin{scope}[yshift=-1cm]
\draw (.5,0) .. controls +(0,-.5) and +(0,-.5) .. +(.5,0);
\draw (1.5,0) to +(0,-.5);
\fill (1.5,-.25) circle(2.5pt);
\draw (2,0) to +(0,-.5);
\end{scope}
\begin{scope}[xshift=3cm,yshift=-1cm]
\draw (.5,0) .. controls +(0,-.5) and +(0,-.5) .. +(.5,0);
\draw (1.5,0) .. controls +(0,-.5) and +(0,-.5) .. +(.5,0);
\end{scope}
\begin{scope}[xshift=6cm,yshift=-1cm]
\draw (.5,0) .. controls +(0,-1) and +(0,-1) .. +(1.5,0);
\draw (1,0) .. controls +(0,-.5) and +(0,-.5) .. +(.5,0);
\end{scope}
\begin{scope}[xshift=9cm,yshift=-1cm]
\draw (.5,0) .. controls +(0,-.5) and +(0,-.5) .. +(.5,0);
\fill (.75,-.35) circle(2.5pt);
\draw (1.5,0) .. controls +(0,-.5) and +(0,-.5) .. +(.5,0);
\fill (1.75,-.35) circle(2.5pt);
\end{scope}
\end{tikzpicture}
\end{eqnarray*}
\end{example}
\begin{defi}
Let $\operatorname{dC}(n)$ to be the set of decorated $(n,k)$-tangles with $1\leq k\leq n$ where all points at the bottom
face are connected to one at the top face and the number of dotted edges plus the number of undotted cups is even.
Elements from $\operatorname{dC}(n)$ are called {\it (even) decorated cup diagrams}.
\end{defi}

The following follows easily from the definitions:
\begin{lem}
\label{cutlemma}
Cutting defines a bijection between  $C(\W^\p)$ and $\operatorname{dC}(n)$.
\end{lem}


We can now easily translate our previous definitions for symmetric cups and weights to decorated cup diagrams. The \textit{diagrammatical weight} $\lambda(w)$ associated to an element $w\in\W^\p$ is obtained from $\alpha_w$ by replacing every {$+$} with a {$\down$} and every {$-$} with an {$\up$}. Now we can glue the weight $\la(w')$ associated to a $w'\in\W^\p$ on top of a decorated cup diagram $\underline{w}$. This gives us a new diagram, denoted $w'\underline{w}$. The \textit{decorated cap diagram} associated to $w\in\W^\p$ is defined to be the vertical reflection of $\underline{w}$ and denoted by $\overline{w}$. Gluing the weight $\la(w')$ below the cap diagram gives us $\overline{w}w'$. Our old definitions translate then into the following:
\begin{defi}
The diagrams $\la\underline{w}$, $\overline{w}\la$, $\op{deg}(\overline{w}\la\underline{w}')$ and are {\it oriented}
if they consists only of {\it oriented decorated cups, caps and edges}, i.e. cups, caps and edges of the form
\begin{eqnarray}
\label{orienteddotted}
\begin{tikzpicture}[thick, scale=0.8]
\node at (-1,-1.2) {degree};
\draw [>->] (0,0) .. controls +(0,-1) and +(0,-1) .. +(1,0) node at +(0.5,-1.2) {0};
\draw [<-<] (2,0) .. controls +(0,-1) and +(0,-1) .. +(1,0) node at +(0.5,-1.2) {1};
\draw [<-<] (4,-0.7) .. controls +(0,1) and +(0,1) .. +(1,0) node at +(0.5,-0.5) {0};
\draw [>->] (6,-0.7) .. controls +(0,1) and +(0,1) .. +(1,0)  node at +(0.5,-0.5) {1};
\draw [>-] (8,0) -- +(0,-0.7) node at +(0,-1.2) {0};
\draw [->] (9,0) -- +(0,-0.7) node at +(0,-1.2) {0};

\node at (-1,-3.2) {degree};
\draw [<->] (0,-2) .. controls +(0,-1) and +(0,-1) .. +(1,0) node at +(0.5,-1.2) {0};
\fill (0.5,-2.75) circle(2.5pt);
\draw [>-<] (2,-2) .. controls +(0,-1) and +(0,-1) .. +(1,0) node at +(0.5,-1.2) {1};
\fill (2.5,-2.76) circle(2.5pt);
\draw [>-<] (4,-2.7) .. controls +(0,1) and +(0,1) .. +(1,0) node at +(0.5,-0.5) {0};
\fill (4.5,-1.93) circle(2.5pt);
\draw [<->] (6,-2.7) .. controls +(0,1) and +(0,1) .. +(1,0)  node at +(0.5,-0.5) {1};
\fill (6.5,-1.95) circle(2.5pt);
\draw [-<] (8,-2) -- +(0,-0.7) node at +(0,-1.2) {0};
\fill (8,-2.35) circle(2.5pt);
\draw [<-] (9,-2) -- +(0,-0.7) node at +(0,-1.2) {0};
\fill (9,-2.35) circle(2.5pt);
\end{tikzpicture}
\end{eqnarray}
respectively. They get assigned a degree as shown. Cups and caps (dotted or undotted) of degree $0$ are called {\it anticlockwise}, cups and caps of degree $1$  {\it clockwise}. We denote by $\op{deg}(\la\underline{w})$ or  $\op{deg}(\overline{w}\la)$ the number of clockwise cups respectively caps in the diagram and by $\op{deg}(\overline{w'}v\underline{w})$ the total number of clockwise cups and caps; in case the diagrams are oriented.

\end{defi}

\begin{rem}
{\rm Instead of $W_\p$ we could also chose the isomorphic parabolic subgroup generated by $s_0, s_2,\ldots s_{n-1}$ and work with $\{+,-\}$-sequences of length $n$ with an odd number of $-$ and the set of {\it odd} decorated cup diagrams instead.
}
\end{rem}

\begin{prop}
\label{dottedcirclediags}
The vector space $E^\p$ can be identified with the span of all oriented circle diagrams  $\overline{w'}v\underline{w}$ for $v,w,w'\in W^\p$. It inherits a natural $\N$-grading from the degree function $\op{deg}$.
\end{prop}

\begin{example}{\rm Let $n=4$ with $W^\p$ as in \eqref{halfseq} and decorated cup diagrams from Example \ref{dottedcups}. The vector space
$E^\p e_w=\operatorname{Hom}_{\O^\p_0}(\oplus_{w'\in\W^\p} P(w'),P(w))$ is for instance $9$-dimensional in case  $w=s_0s_2$; with diagram basis
\begin{equation}
\begin{tikzpicture}[thick, scale=0.8]
\begin{scope} [xshift=4.5cm]
\draw (8,0) node{$\operatorname{deg}=0$};
\draw[<-] (.5,0) to +(0,-.5);
\fill (.5,-.3) circle(2.5pt);
\draw (1,0)[>->] .. controls +(0,-.5) and +(0,-.5) .. +(.5,0);
\draw[>-] (2,0) to +(0,-.5);
\draw(.5,-.1) to +(0,.6);
\fill (.5,.25) circle(2.5pt);
\draw (1,-.1) .. controls +(0,.5) and +(0,.5) .. +(.5,0);
\draw(2,-.1) to +(0,.6);
\end{scope}
\begin{scope} [yshift=-1.5cm]
\draw (12.5,0) node{$\operatorname{deg}=1$};
\draw[<-] (.5,0) to +(0,-.5);
\fill (.5,-.3) circle(2.5pt);
\draw (1,0)[>->] .. controls +(0,-.5) and +(0,-.5) .. +(.5,0);
\draw[>-] (2,0) to +(0,-.5);
\draw(.5,-.1) to +(0,.6);
\fill (.5,.25) circle(2.5pt);
\draw(1,-.1) to +(0,.6);
\draw (1.5,-.1) .. controls +(0,.5) and +(0,.5) .. +(.5,0);
\begin{scope} [xshift=3cm]
\draw[<-] (.5,0) to +(0,-.5);
\fill (.5,-.3) circle(2.5pt);
\draw (1,0)[>->] .. controls +(0,-.5) and +(0,-.5) .. +(.5,0);
\draw[>-] (2,0) to +(0,-.5);
\draw (.5,-.1) .. controls +(0,.5) and +(0,.5) .. +(.5,0);
\draw(1.5,-.1) to +(0,.6);
\fill (1.5,.25) circle(2.5pt);
\draw(2,-.1) to +(0,.6);
\end{scope}
\begin{scope} [xshift=6cm]
\draw[<-] (.5,0) to +(0,-.5);
\fill (.5,-.3) circle(2.5pt);
\draw (1,0)[>->] .. controls +(0,-.5) and +(0,-.5) .. +(.5,0);
\draw[>-] (2,0) to +(0,-.5);
\draw (1,-.1) .. controls +(0,.5) and +(0,.5) .. +(.5,0);
\draw (.5,-.1) .. controls +(0,1) and +(0,1) .. +(1.5,0);
\end{scope}
\begin{scope} [xshift=9cm]
\draw[<-] (.5,0) to +(0,-.5);
\fill (.5,-.3) circle(2.5pt);
\draw (1,0)[<-<] .. controls +(0,-.5) and +(0,-.5) .. +(.5,0);
\draw[>-] (2,0) to +(0,-.5);
\draw (.5,-.1) .. controls +(0,.5) and +(0,.5) .. +(.5,0);
\fill (.75,.25) circle(2.5pt);
\draw(1.5,-.1) to +(0,.6);
\draw(2,-.1) to +(0,.6);
\end{scope}
\end{scope}

\begin{scope} [yshift=-3cm, xshift=1.5cm]
\draw (11,0) node{$\operatorname{deg}=2$};
\draw[<-] (.5,0) to +(0,-.5);
\fill (.5,-.3) circle(2.5pt);
\draw (1,0)[>->] .. controls +(0,-.5) and +(0,-.5) .. +(.5,0);
\draw[>-] (2,0) to +(0,-.5);
\draw (.5,-.1) .. controls +(0,.5) and +(0,.5) .. +(.5,0);
\draw (1.5,-.1) .. controls +(0,.5) and +(0,.5) .. +(.5,0);
\begin{scope} [xshift=3cm]
\draw[<-] (.5,0) to +(0,-.5);
\fill (.5,-.3) circle(2.5pt);
\draw (1,0)[<-<] .. controls +(0,-.5) and +(0,-.5) .. +(.5,0);
\draw[>-] (2,0) to +(0,-.5);
\draw(.5,-.1) to +(0,.6);
\fill (.5,.25) circle(2.5pt);
\draw (1,-.1) .. controls +(0,.5) and +(0,.5) .. +(.5,0);
\draw(2,-.1) to +(0,.6);
\end{scope}

\begin{scope}[xshift=6cm]
\draw[<-] (.5,0) to +(0,-.5);
\fill (.5,-.3) circle(2.5pt);
\draw (1,0)[>->] .. controls +(0,-.5) and +(0,-.5) .. +(.5,0);
\draw[>-] (2,0) to +(0,-.5);
\draw (.5,-.1) .. controls +(0,.5) and +(0,.5) .. +(.5,0);
\fill (.75,.27) circle(2.5pt);
\draw (1.5,-.1) .. controls +(0,.5) and +(0,.5) .. +(.5,0);
\fill (1.75,.27) circle(2.5pt);
\end{scope}
\end{scope}
\begin{scope} [yshift=-4.5cm, xshift=4.5cm]
\draw (8,0) node{$\operatorname{deg}=3$};
\draw[<-] (.5,0) to +(0,-.5);
\fill (.5,-.3) circle(2.5pt);
\draw (1,0)[<-<] .. controls +(0,-.5) and +(0,-.5) .. +(.5,0);
\draw[>-] (2,0) to +(0,-.5);
\draw (1,-.1) .. controls +(0,.5) and +(0,.5) .. +(.5,0);
\draw (.5,-.1) .. controls +(0,1) and +(0,1) .. +(1.5,0);
\end{scope}
\end{tikzpicture}
\end{equation}
The decorated cup diagram is our fixed $\underline{w}$, the decorated cap diagrams vary through $\overline{w'}$ indicating the corresponding summand. The degrees match with the $q$ powers in the Kazhdan-Lusztig polynomials. It is easy to determine a diagram basis for $E=E^\p$ and show that $\op{dim}E^\p=67$ with graded Poincare polynomials
\begin{equation*}
\begin{array}[t]{lll}
P_{E e_{++++}}=1+q+q^2,&P_{E e_{--++}}= 1+3q+2q^2+q^3,\\
P_{E e_{-+-+}}=1+4q+3q^2+q^3,& P_{E e_{-++-}}=1+2q+3q^2+q^3,\\
P_{E e_{+--+}}=1+2q+3q^2+q^3,& P_{E e_{+-+-}}=1+3q+3q^2+3q^3+q^4,\\
P_{E e_{++--}}=1+3q+5q^2+3q^3+q^4,& P_{E e_{----}}=1+2q+4q^2+2q^3+q^4.
\end{array}
\end{equation*}}
\end{example}

\subsection{The action of $\H$ on $\Nmod$ diagrammatically}
%
%
%
Although we have an explicit description of the Kazhdan-Lusztig basis, the action of $\H$ on a basis element is still not easy to calculate. But there is an easy way to multiply some $\underline{N}_w$ with a $C_{s_i}$ in a diagrammatic way.\\
Consider the free $\L$-module $\Nmod_t$ with basis $\operatorname{dC}(n)$. Define an action of  $\mathbb{DT}_n$ as follows: a generator $T=e_i$ acts on a basis vector $b$ from $\Nmod_t$ by putting $T$ on top of $b$, concatenate and then remove all caps and internal circles according to the following rules (giving rise to an $\L$-multiple of a basis vector from $\Nmod_t$ or zero):
\begin{equation}
\label{actionrel}
\begin{tikzpicture}[thick,scale=0.85]
\draw node{\it i.)} +(1,0) circle(10pt) +(2.3,0) node{$=q+q^{-1}$};
\draw (3.8,0) node{\it ii.)} +(.7,-.35) to +(.7,.45);
\fill (4.5,-.05) circle(2.5pt) +(0,.2) circle(2.5pt);
\draw (4.9,0) node {$=$} +(.3,-.35) to +(.3,.45);
\draw (6.2,0) node{\it iii.)} +(1.05,0) circle(10pt) +(1.8,0) node{$=0$};
\fill (6.9,0) circle(2.5pt);
\draw (9.2,0) node{\it iv.)};
\draw (9.9,-.35) node{$\otimes$} .. controls +(0,1) and +(0,1) .. +(.5,0)  node{$\otimes$} +(1,.35) node{$=0$};
\draw (12,0) node{\it v.)};
\draw (12.6,-.35) node{$\otimes$} .. controls +(0,1) and +(0,1) .. +(.5,0)  node{$\otimes$} +(1,.35) node{$=1$};
\fill (12.85,.37) circle(2.5pt);
\end{tikzpicture}
\end{equation}

(The last two relations hold for caps only, the symbols $\otimes$ indicate that the intersection points of the tangle with the lower face of the rectangle.)

\begin{rem}{\rm The first relation states that any circle with no decorations can be removed and, in turn, the diagram is multiplied by $(q+q^{-1})$. The second relation means that from any circle or line an even number of decorations can be removed. The third relation says that the presence of a circle with only one decoration multiplies the whole diagram by $0$ and hence annihilates the whole diagram. An undotted cap annihilates the diagram, whereas a cap decorated by one dot can just be removed.
}
\end{rem}

\begin{example}
\label{quiver}
{\rm In the situation of Example~\ref{dottedcups}, the action of the $e_i$ is given as follows (where we omitted the $e_i$'s which act by zero and indicated multiplication by $(q+q^{1})$ by the circled dashed arrows).
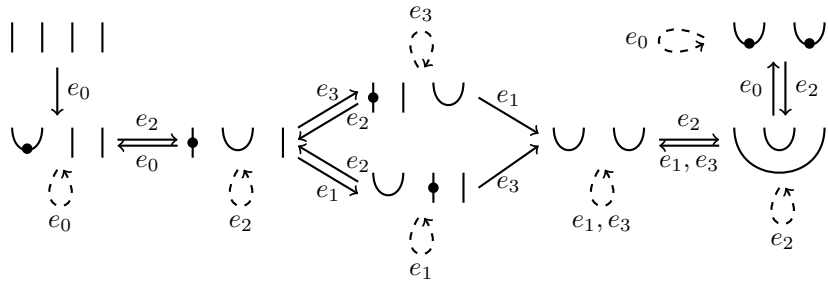
\begin{eqnarray*}
\begin{tikzpicture}[thick,scale=.8]

\begin{scope}[xshift=3cm,yshift=1.75cm]
\draw (.5,0) to +(0,-.5);
\draw (1,0) to +(0,-.5);
\draw (1.5,0) to +(0,-.5);
\draw (2,0) to +(0,-.5);
\end{scope}

\draw[->] (4.25,1) to node[right]{$e_0$} (4.25,.2) ;

\begin{scope}[xshift=3cm]
\draw (0.5,0) .. controls +(0,-.5) and +(0,-.5) .. +(.5,0);
\fill (.75,-.35) circle(2.5pt);
\draw (1.5,0) to +(0,-.5);
\draw (2,0) to +(0,-.5);
\end{scope}

\draw[->,dashed] (4.25,-.7) .. controls +(-.5,-.8) and +(.5,-.8) .. +(.1,0) node[midway,below]{$e_0$};
\draw[->] (5.25,-.2) to node[above]{$e_2$} (6.25,-.2);
\draw[<-] (5.25,-.3) to node[below]{$e_0$} (6.25,-.3);

\begin{scope}[xshift=6cm]
\draw (.5,0) to +(0,-.5);
\fill (.5,-.25) circle(2.5pt);
\draw (1,0) .. controls +(0,-.5) and +(0,-.5) .. +(.5,0);
\draw (2,0) to +(0,-.5);
\end{scope}

\draw[->,dashed] (7.25,-.7) .. controls +(-.5,-.8) and +(.5,-.8) .. +(.1,0) node[midway,below]{$e_2$};
\draw[->,dashed] (10.25,-1.5) .. controls +(-.5,-.8) and +(.5,-.8) .. +(.1,0) node[midway,below]{$e_1$};
\draw[->,dashed] (10.25,1) .. controls +(-.5,.8) and +(.5,.8) .. +(.1,0) node[midway,above]{$e_3$};

\draw[->] (8.25,0) to node[above]{$e_3$} (9.25,.6);
\draw[<-] (8.25,-.2) to (9.25,.4) node[below]{$e_2$};

\draw[<-] (8.25,-.3) to (9.25,-.9) node[above]{$e_2$};
\draw[->] (8.25,-.5) to node[below]{$e_1$} (9.25,-1.1);

\begin{scope}[xshift=9cm,yshift=.75cm]
\draw (.5,0) to +(0,-.5);
\draw (1,0) to +(0,-.5);
\fill (.5,-.25) circle(2.5pt);
\draw (1.5,0) .. controls +(0,-.5) and +(0,-.5) .. +(.5,0);
\end{scope}

\begin{scope}[xshift=9cm,yshift=-.75cm]
\draw (.5,0) .. controls +(0,-.5) and +(0,-.5) .. +(.5,0);
\draw (1.5,0) to +(0,-.5);
\fill (1.5,-.25) circle(2.5pt);
\draw (2,0) to +(0,-.5);
\end{scope}

\draw[->] (11.25,.5) to node[above]{$e_1$} (12.25,-.1);
\draw[->] (11.25,-1) to node[below]{$e_3$} (12.25,-.3);
\draw[->,dashed] (13.25,-.7) .. controls +(-.5,-.8) and +(.5,-.8) .. +(.1,0) node[midway,below]{$e_1,e_3$};

\begin{scope}[xshift=12cm]
\draw (.5,0) .. controls +(0,-.5) and +(0,-.5) .. +(.5,0);
\draw (1.5,0) .. controls +(0,-.5) and +(0,-.5) .. +(.5,0);
\end{scope}

\draw[->,dashed] (16.25,-1) .. controls +(-.5,-.8) and +(.5,-.8) .. +(.1,0) node[midway,below]{$e_2$};
\draw[->] (14.25,-.2) to node[above]{$e_2$} (15.25,-.2);
\draw[<-] (14.25,-.3) to node[below]{$e_1,e_3$} (15.25,-.3);

\begin{scope}[xshift=15cm]
\draw (.5,0) .. controls +(0,-1) and +(0,-1) .. +(1.5,0);
\draw (1,0) .. controls +(0,-.5) and +(0,-.5) .. +(.5,0);
\end{scope}

\draw[->,dashed] (15,1.5) .. controls +(-1,.5) and +(-1,-.5) .. +(0,-.1) node[midway,left]{$e_0$};
\draw[->] (16.15,.2) to node[left]{$e_0$} (16.15,1.1);
\draw[<-] (16.35,.2) to node[right]{$e_2$} (16.35,1.1);

\begin{scope}[xshift=15cm,yshift=1.75cm]
\draw (.5,0) .. controls +(0,-.5) and +(0,-.5) .. +(.5,0);
\fill (.75,-.35) circle(2.5pt);
\draw (1.5,0) .. controls +(0,-.5) and +(0,-.5) .. +(.5,0);
\fill (1.75,-.35) circle(2.5pt);
\end{scope}
\end{tikzpicture}
\end{eqnarray*}
}
\end{example}

\begin{lem}
\label{lem:eventoeven}
The above rules define an action of  $\mathbb{DT}_n$ on $\Nmod_t$.
\end{lem}

\begin{proof}
It is enough to show that the an even decorated cup is either sent to zero or a multiple of an even decorated cup by each of the generators $e_i$, $0\leq i\leq n-1$. The following list shows all the possible configurations of cups without dots involving the points $i$ and $i+1$ (indicated by little boxes):
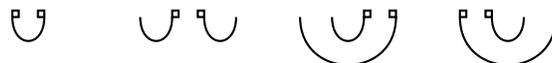
\begin{center}
\begin{tikzpicture}[thick,scale=0.85]

\draw (1,0) rectangle (1.1,.1);
\draw (1,0) .. controls +(0,-.5) and +(0,-.5) .. +(.5,0) rectangle +(.4,.1);

\begin{scope}[xshift=2cm]
\draw (1,0) .. controls +(0,-.5) and +(0,-.5) .. +(.5,0) rectangle +(.6,.1);
\draw (2,0) rectangle (1.9,.1);
\draw (2,0) .. controls +(0,-.5) and +(0,-.5) .. +(.5,0);
\end{scope}

\begin{scope}[xshift=4.5cm]
\draw (1,0) .. controls +(0,-1) and +(0,-1) .. +(1.5,0) rectangle +(1.4,.1);
\draw (1.5,0) .. controls +(0,-.5) and +(0,-.5) .. +(.5,0) rectangle +(.6,.1);
\end{scope}

\begin{scope}[xshift=7cm]
\draw (1,0) rectangle (1.1,.1);
\draw (1.5,0) rectangle (1.4,.1);
\draw (1,0) .. controls +(0,-1) and +(0,-1) .. +(1.5,0);
\draw (1.5,0) .. controls +(0,-.5) and +(0,-.5) .. +(.5,0);
\end{scope}
\end{tikzpicture}
\end{center}

Applying $e_i$ for $i\not=0$ multiplies with $(q+q^{-1})$ or creates from the two involved undotted cups two new undotted cups. Applying $e_0$ gives zero in the first case and replaces two undotted cups by two dotted cups in the last case. The other configurations are impossible. The following list shows all the possible configurations involving the points $i$ and $i+1$ with at least one dot:

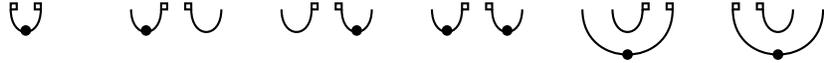
\begin{center}
\begin{tikzpicture}[thick,scale=0.8]

\draw (1,0) rectangle (1.1,.1); \draw (1,0) .. controls
+(0,-.5) and +(0,-.5) .. +(.5,0) rectangle +(.4,.1); \fill
(1.25,-.35) circle(2.5pt);

\begin{scope}[xshift=2cm]
\draw (1,0) .. controls +(0,-.5) and +(0,-.5) .. +(.5,0)
rectangle +(.6,.1); \draw (2,0) rectangle (1.9,.1); \draw (2,0)
.. controls +(0,-.5) and +(0,-.5) .. +(.5,0); \fill (1.25,-.35)
circle(2.5pt);
\end{scope}

\begin{scope}[xshift=4.5cm]
\draw (1,0) .. controls +(0,-.5) and +(0,-.5) .. +(.5,0)
rectangle +(.6,.1); \draw (2,0) rectangle (1.9,.1); \draw (2,0)
.. controls +(0,-.5) and +(0,-.5) .. +(.5,0); \fill (2.25,-.35)
circle(2.5pt);
\end{scope}

\begin{scope}[xshift=7cm]
\draw (1,0) .. controls +(0,-.5) and +(0,-.5) .. +(.5,0)
rectangle +(.6,.1); \draw (2,0) rectangle (1.9,.1); \draw (2,0)
.. controls +(0,-.5) and +(0,-.5) .. +(.5,0); \fill (1.25,-.35)
circle(2.5pt); \fill (2.25,-.35) circle(2.5pt);
\end{scope}

\begin{scope}[xshift=9.5cm]
\draw (1,0) .. controls +(0,-1) and +(0,-1) .. +(1.5,0)
rectangle +(1.4,.1); \draw (1.5,0) .. controls +(0,-.5) and
+(0,-.5) .. +(.5,0) rectangle +(.6,.1); \fill (1.75,-.75)
circle(2.5pt);
\end{scope}

\begin{scope}[xshift=12cm]
\draw (1,0) rectangle (1.1,.1); \draw (1.5,0) rectangle
(1.4,.1); \draw (1,0) .. controls +(0,-1) and +(0,-1) ..
+(1.5,0); \draw (1.5,0) .. controls +(0,-.5) and +(0,-.5) ..
+(.5,0); \fill (1.75,-.75) circle(2.5pt);
\end{scope}
\end{tikzpicture}
\end{center}

Applying $e_i$ for $i\not=0$ gives zero or creates from the two involved cups two new cups, from one dotted and one undotted we obtain again one dotted and one undotted, whereas the two dotted give two undotted cups. Finally one has to consider the situations, where not only cups are involved:

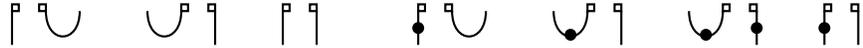
\begin{center}
\begin{tikzpicture}[thick, scale=0.9]

\draw (1,0) rectangle +(.1,.1);
\draw (1,0) to +(0,-.5);
\draw (1.5,0) rectangle +(-.1,.1);
\draw (1.5,0) .. controls +(0,-.5) and +(0,-.5) .. +(.5,0);

\begin{scope}[xshift=2cm]
\draw (1,0) .. controls +(0,-.5) and +(0,-.5) .. +(.5,0) rectangle +(.6,.1);
\draw (2,0) rectangle +(-.1,.1);
\draw (2,0) to +(0,-.5);
\end{scope}

\begin{scope}[xshift=4cm]
\draw (1,0) rectangle +(.1,.1);
\draw (1,0) to +(0,-.5);
\draw (1.5,0) rectangle +(-.1,.1);
\draw (1.5,0) to +(0,-.5);
\end{scope}

\begin{scope}[xshift=6cm]
\draw (1,0) rectangle +(.1,.1);
\draw (1,0) to +(0,-.5);
\fill (1,-.25) circle(2.5pt);
\draw (1.5,0) rectangle +(-.1,.1);
\draw (1.5,0) .. controls +(0,-.5) and +(0,-.5) .. +(.5,0);
\end{scope}

\begin{scope}[xshift=8cm]
\draw (1,0) .. controls +(0,-.5) and +(0,-.5) .. +(.5,0) rectangle +(.6,.1);
\draw (2,0) rectangle +(-.1,.1);
\draw (2,0) to +(0,-.5);
\fill (1.25,-.35) circle(2.5pt);
\end{scope}

\begin{scope}[xshift=10cm]
\draw (1,0) .. controls +(0,-.5) and +(0,-.5) .. +(.5,0) rectangle +(.6,.1);
\draw (2,0) rectangle +(-.1,.1);
\draw (2,0) to +(0,-.5);
\fill (1.25,-.35) circle(2.5pt);
\fill (2,-.25) circle(2.5pt);
\end{scope}

\begin{scope}[xshift=12cm]
\draw (1,0) rectangle +(.1,.1);
\draw (1,0) to +(0,-.5);
\draw (1.5,0) rectangle +(-.1,.1);
\draw (1.5,0) to +(0,-.5);
\fill (1,-.25) circle(2.5pt);
\end{scope}

\end{tikzpicture}
\end{center}
Applying $e_i$ for $i\not=0$ keeps on undotted cup, multiplies by $(q+q^{-1})$, annihilates, creates a dotted edge and a dotted cup, annihilates, annihilates or creates a undotted cup respectively. For $e_0$, the configuration is impossible for the second, fifth and sixth diagram. Otherwise, applying $e_0$ creates a dotted edge and a dotted cup, creates a dotted cup, creates a dotted cup and an undotted edge, and annihilates respectively. Hence in each case we again obtain an even decorated cup diagram. The lemma follows.
\end{proof}

Because of Lemma \ref{cutlemma}, the module $\Nmod_t$ is isomorphic to $\Nmod$ as an $\L$-module by sending $\underline{N}_w$ to the decorated cup diagram $\underline{w}$. We fix this isomorphism, $\Phi$.

\begin{thm}
\label{cupcap}
The following diagram commutes:
	\begin{eqnarray}
\label{commdiag}
	\begin{pspicture}(0,0)(4,4.7)
	\rput(0,0){$\Nmod$}
	\rput(4,0){$\Nmod_t$}
	\rput(0,4){$\Nmod$}
	\rput(4,4){$\Nmod_t$}
	\psline{->}(1,0)(3,0)
	\psline{->}(1,4)(3,4)
	\psline{->}(0,3)(0,1)
	\psline{->}(4,3)(4,1)
	\rput(2,0.5){$\Phi$}
	\rput(2,4.5){$\Phi$}
	\rput(-0.7,2){$C_{s_i}$}
	\rput(4.6,2){$e_i$}
	\end{pspicture}
\end{eqnarray}
 Moreover, the $\mathbb{DT}_n$-action factors through $\operatorname{\widehat{TL}}(\W)$, the quotient of $\mathbb{DT}_n$ modulo the relations (i)-(iii) in \eqref{actionrel} in case $n$ is odd and modulo the additional subspace $I$ spanned by all diagrams without edges and an odd number of undotted cups (or equivalently of undotted caps) in case $n$ os even. The resulting $\operatorname{\widehat{TL}}(\W)$-action is faithful.
\end{thm}

	\begin{rem}\label{rem:TL}
	The first two relations coincide with the first two relations from \eqref{fig:TL-Rel}. Relation $(iii)$ implies the remaining one, and so the $\H$-action factors through the generalized Temperley-Lieb algebra $\operatorname{TL}(\W)$.
	\end{rem}

%
%
%
%
The proof of Theorem \ref{cupcap} heavily relies on the {\it locality} of the action, that means to understand the action on cup diagrams, usually it is enough to consider those parts of the cup diagram that change under the action. To be able to handle those we define the relative Kazhdan-Lusztig elements.
	\begin{defi}
\label{subKL}
	A {\it cup subdiagram} is any part of the cup diagram  $C(w)$, $w\in W^\p$, which itself is a diagram of cups. Let $SC(w)$ be a symmetric cup subdiagram of $C(w)$. The \textit{relative Kazhdan-Lusztig element} associated to $SC(w)$ is $\sum q^{\frac{\operatorname{cl}(w'SC(w))}{2}}N_{w'}$ where the sum runs over all different subweights $w'$ such that $SP(w)$ is oriented. A {\it subweight} is obtained from a weight by taking the subset
	$$P=\{i\mid\mbox{$i$ is the starting or ending point of a cup in $SP(w)$}\}\subset\{1,\hdots,2n\}$$
	and considering only the orientations at these points.
	\end{defi}
	\begin{rem}
\begin{enumerate}
\item Given a symmetric cup subdiagram different weights may lead to the same subweights. In Definition \ref{subKL} we sum over different subweights. Every subweight appears only once in our sum even when there is more than one weight having this subweight.
	\item To denote weights we sometimes wrote just the part between $1$ and $n$ knowing the rest by antisymmetry and the frozen parts. We do the same with our subweights, writing only those $i$ that lie between $1$ and $n$, as usual indicated by a $|$ at the start and $]$ at the end of the sequence.
	\item
	If we write a cup diagram as a disjoint union of symmetric cup subdiagrams we are able to calculate our Kazhdan-Lusztig basis element from the relative Kazhdan-Lusztig elements. The Kazhdan-Lusztig polynomial in front of a weight is the product of the coefficients of the subweights in the relative Kazhdan-Lusztig elements to which the weight restricts. Indeed, by  Theorem \ref{main_theorem} we have to count clockwise oriented cups in our cup diagram. The latter is a disjoint union of symmetric cup subdiagrams, so the total number of clockwise oriented cups is the sum of clockwise oriented cups in each cup subdiagram. Multiplying the coefficients simply adds up the numbers of clockwise oriented cups in each cup subdiagram or gives $0$ if one of the cups in any of the cup subdiagrams is not oriented.
\end{enumerate}	
\end{rem}
We can consider the action of the $C_{s_i}$'s on our subweights even though our subweights need not lie in some $\W^\p$ for some smaller $n$. Given a cup subdiagram $SP(w)$ we have a map
$$\{\mbox{weights}\}\twoheadrightarrow\{\mbox{subweights associated to $SP(w)$}\}.$$
From \eqref{N} we see that the action of some $C_{s_i}$ on $N_w$ is described using only $ws_i$ and $w$. But $s_i$ operates on the weight $w$ by changing at most two points, namely $i$ and $i+1$ for $1\leq i\leq n-1$, or $1$ and $2$ if $i=0$. Hence we can define the action of $C_{s_i}$ on a subweight in case it contains both points by taking a weight that restricts to the subweight, operate with $C_{s_i}$ on this weight according to the multiplication rules, and then restrict the result again.

		\begin{proof}[Proof of Theorem \ref{cupcap}]
		We repeatedly use Theorem \ref{main_theorem} without explicitly referring to it. Also, throughout the whole proof, we are going to consider only the `relevant' parts of the diagrams, meaning the cup subdiagrams consisting of those parts that are involved in the change of the cup diagram. To prove the theorem we have to consider three cases:
		\begin{enumerate}[(I)]
			\item $ws_i\in\W^\p$ and $l(ws_i)>l(w)$
			\item $ws_i\in\W^\p$ and $l(ws_i)<l(w)$
			\item $ws_i\notin\W^\p$
		\end{enumerate}

{\bf Case (I)}: We know that $\underline{N}_wC_{s_i}=\underline{N}_{ws_i}$. We have to go through the cases \eqref{a1}, \eqref{a2}, \eqref{b1}. In all the following pictures, above the gray line we have the involved part of the tangle $e_i$ while below the relevant part of the cup diagram $T(w)$ is shown. We give the example $1\leq i\leq n-1$ in case \eqref{a1} and omit the other cases. We have
						$$
						
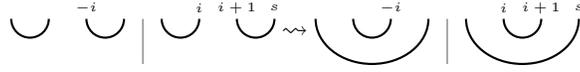
\begin{pspicture}(-3.5,0)(3.5,-0.5)
						\psline[linecolor=gray, linewidth=0.5pt](0,0)(0,-1.2)
						\psellipticarc(1,0)(0.5,0.5){180}{0}
						\psellipticarc(3,0)(0.5,0.5){180}{0}
						\rput(1.5,0.3){\tiny$i$}
						\rput(2.5,0.3){\tiny$i+1$}
						\rput(3.5,0.3){\tiny$s$}
						\psellipticarc(-1,0)(0.5,0.5){180}{0}
						\psellipticarc(-3,0)(0.5,0.5){180}{0}
						\rput(-1.5,0.3){\tiny$-i$}
						\end{pspicture}
						\leadsto
						\begin{pspicture}[shift=-1](-4,0)(3,-1.5)
						\psline[linecolor=gray, linewidth=0.5pt](-0.5,0)(-0.5,-1.2)
						\psellipticarc(1.5,0)(0.5,0.5){180}{0}
						\psellipticarc(1.5,0)(1.5,1.2){180}{0}
						\rput(1,0.3){\tiny $i$}
						\rput(2,0.3){\tiny $i+1$}
						\rput(3,0.3){\tiny $s$}
						\psellipticarc(-2.5,0)(0.5,0.5){180}{0}
						\psellipticarc(-2.5,0)(1.5,1.2){180}{0}
						\rput(-2,0.3){\tiny $-i$}
						\end{pspicture}$$
						Depending on $s$ being at most $n$ or greater $n$ we have the following situation
						\begin{itemize}
						\item[$s\leq n$:]
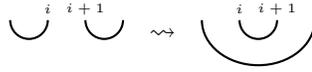
							$$
							\begin{pspicture}(0,0)(4,-0.5)
							\psellipticarc(1,0)(0.5,0.5){180}{0}
							\psellipticarc(3,0)(0.5,0.5){180}{0}
							\rput(1.5,0.3){\tiny$i$}
							\rput(2.5,0.3){\tiny$i+1$}
							\end{pspicture}
							\leadsto
							\begin{pspicture}[shift=-1](-0.5,0)(3,-1.5)
							\psellipticarc(1.5,0)(0.5,0.5){180}{0}
							\psellipticarc(1.5,0)(1.5,1.2){180}{0}
							\rput(1,0.3){\tiny $i$}
							\rput(2,0.3){\tiny $i+1$}
							\end{pspicture}
							$$
							$$\mbox{On the other hand, multiplication by $e_i$ gives us}$$
							$$
							
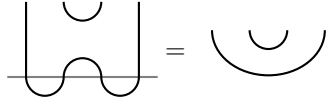
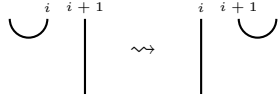
\begin{pspicture}[shift=-1](0,2)(4,-0.5)
							\psline[linecolor=gray](0,0)(4,0)
							\psellipticarc(1,0)(0.5,0.5){180}{0}
							\psellipticarc(3,0)(0.5,0.5){180}{0}
							\psline(0.5,0)(0.5,2)
							\psline(3.5,0)(3.5,2)
							\psellipticarc(2,0)(0.5,0.5){0}{180}
							\psellipticarc(2,2)(0.5,0.5){180}{0}
							\end{pspicture}
							=
							 \begin{pspicture}[shift=-0.75](-0.5,0)(3,-1.5)
							\psellipticarc(1.5,0)(0.5,0.5){180}{0}
							\psellipticarc(1.5,0)(1.5,1.2){180}{0}
							\end{pspicture}
							$$
						\item[$s>n$:]
							$$
							\begin{pspicture}[shift=-1](0,0)(3.5,-2)
							\psellipticarc(1,0)(0.5,0.5){180}{0}
							\psline(2.5,0)(2.5,-2)
							\rput(1.5,0.3){\tiny$i$}
							\rput(2.5,0.3){\tiny$i+1$}
							\end{pspicture}
							\leadsto
							\begin{pspicture}[shift=-1](-0.5,0)(3,-2)
							\psellipticarc(2,0)(0.5,0.5){180}{0}
							\psline(0.5,0)(0.5,-2)
							\rput(0.5,0.3){\tiny $i$}
							\rput(1.5,0.3){\tiny $i+1$}
							\end{pspicture}$$
							$$\mbox{Multiplication by $e_i$ gives us}$$
							 $$
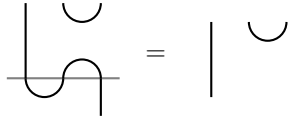
\begin{pspicture}[shift=-1.5](0,2.5)(3.5,-1)
							\psline[linecolor=gray](0,0)(3,0)
							\psellipticarc(1,0)(0.5,0.5){180}{0}
							\psline(2.5,0)(2.5,-1)
							\psline(0.5,0)(0.5,2)
							\psellipticarc(2,0)(0.5,0.5){0}{180}
							\psellipticarc(2,2)(0.5,0.5){180}{0}
							\end{pspicture}
							=
							\begin{pspicture}[shift=-1](-0.5,0)(3,-2)
							\psellipticarc(2,0)(0.5,0.5){180}{0}
							\psline(0.5,0)(0.5,-2)
							\end{pspicture}
							$$
						\end{itemize}
	
{\bf Case (II):} We find $w'\in W^\p$ and $s_i\in\S$ such that $w=w's_i$ for some $w'\in\W^\p$. Hence $\underline{N}_w=\underline{N}_{w'}C_{s_i}$, an so  $\underline{N}_wC_{s_i}=\underline{N}_{w'}C_{s_i}^2=(q+q^{-1})\underline{N}_w$, because $C_{s_i}^2=(q+q^{-1})C_{s_i}$.
For our $\{+,-\}$-sequence $l(w_i)<l(w)$ means that we have for $i\not=0$ a {$+$} at place $i$ and a {$-$} at place $i+1$, which get exchanged by $s_i$, and for $i=0$ two {$-$}'s at places $1$ and $2$ which are exchanged for two {$+$}'s. So in $T(w)$ we have a cup going from $i$ to $i+1$, respectively two linked cups going from $-1$ to $2$ and from $-2$ to $1$.   Multiplying by $e_i$ gives us
						$$\begin{pspicture}[shift=-0.75](0,0)(1.5,2)
						\psline[linecolor=gray](-0.5,0.5)(1.5,0.5)
						\pscircle(0.5,0.5){0.5}
						\psellipticarc(0.5,2)(0.5,0.5){180}{0}			
						\end{pspicture}
						\stackrel{(i)}{=}
						(q+q^{-1})\cdot
						\begin{pspicture}[shift=-0.5](-0.75,-0.5)(0.5,0.5)
						\psellipticarc(0,0.5)(0.5,0.5){180}{0}			
						\end{pspicture}$$
$$\begin{pspicture}[shift=-0.75](0,0)(1.5,2)
						\psline[linecolor=gray](-0.5,0.5)(1.5,0.5)
						\pscircle(0.5,0.5){0.5} \pscircle*(0.5,0){0.2} \pscircle*(0.5,1){0.2} \pscircle*(0.5,1.5){0.2}
						\psellipticarc(0.5,2)(0.5,0.5){180}{0}			
						\end{pspicture}
						\stackrel{(ii)}{=}
						\begin{pspicture}[shift=-0.75](-0.5,0)(1.5,2)
						\psline[linecolor=gray](-0.5,0.5)(1.5,0.5)
						\pscircle(0.5,0.5){0.5} \pscircle*(0.5,1.5){0.2}
						\psellipticarc(0.5,2)(0.5,0.5){180}{0}			
						\end{pspicture}
						\stackrel{(i)}{=}
						(q+q^{-1})\cdot
						\begin{pspicture}[shift=-0.5](-0.75,-0.5)(0.5,0.5)
						\psellipticarc(0,0.5)(0.5,0.5){180}{0} \pscircle*(0,0){0.2}		
						\end{pspicture}$$
						So our tangle stays the same but is multiplied by $(q+q^{-1})$ which is exactly what we wanted.

{\bf Case (III):} This is the most tricky case; it is not easy to describe the action of a $C_{s_i}$ on $\underline{N}_w$. We argue in terms of oriented cup subdiagrams and relative Kazhdan-Lusztig elements, distinguishing four cases:
			\begin{enum}
			\item $1\leq i\leq n-1$ and pluses at $i$ and $i+1$
			\item $1\leq i\leq n-1$ and minuses at $i$ and $i+1$
			\item $i=0$ and {$+$} at $1$, {$-$} at $2$
			\item $i=0$ and {$-$} at $1$, {$+$} at $2$
			\end{enum}
			We do the first case, the others are similar.
				\begin{enum}
				\item $1\leq i\leq n-1$ and pluses at $i$ and $i+1$:
					In our full cup diagram the relevant parts look like this:
					$$
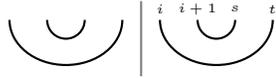
\begin{pspicture}[shift=-1](-4,0)(3,-1.5)
					\psellipticarc(1.5,0)(0.5,0.5){180}{0}
					\psellipticarc(1.5,0)(1.5,1.2){180}{0}
					\rput(0,0.3){\tiny $i$}
					\rput(1,0.3){\tiny $i+1$}
					\rput(2,0.3){\tiny $s$}
					\rput(3,0.3){\tiny $t$}
					\psellipticarc(-2.5,0)(0.5,0.5){180}{0}
					\psellipticarc(-2.5,0)(1.5,1.2){180}{0}
					\psline[linecolor=gray](-0.5,0.5)(-0.5,-1.5)
					\end{pspicture}$$
					The relative Kazhdan-Lusztig element associated to this is
					 $$K=N_{|\down\down\up\up]}+vN_{|\down\up\down\up]}+vN_{|\up\down\up\down]}+v^2N_{|\up\up\down\down]}.$$
					Depending on the endings of the cups, some terms become $0$ because of restrictions on the weights. Multiplying the partial KL-polynomial by $C_{s_i}$ the first and the last term become $0$. The second one gets shorter, the third one longer. This leaves us with
					 $$K_{new}=N_{|\down\up\down\up]}+qN_{|\up\down\down\up]}+qN_{|\down\up\up\down]}+q^2N_{|\up\down\up\down]}.$$
					This is the relative KL-element associated to the cup subdiagram				
						$$
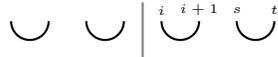
\begin{pspicture}(0,0.5)(4,-0.5)
						\psellipticarc(1,0)(0.5,0.5){180}{0}
						\psellipticarc(3,0)(0.5,0.5){180}{0}
						\rput(0.5,0.3){\tiny$i$}
						\rput(1.5,0.3){\tiny$i+1$}
						\rput(2.5,0.3){\tiny $s$}
						\rput(3.5,0.3){\tiny $t$}
						\psline[linecolor=gray](0,0.5)(0,-1)
						\psellipticarc(-1,0)(0.5,0.5){180}{0}
						\psellipticarc(-3,0)(0.5,0.5){180}{0}
						\end{pspicture}$$
					Depending on the values of $s$ and $t$ there are three different situations:
					\begin{enumerate}[]
					\item {\bf $s,t\leq n$:} Then the part in our decorated tangle is the same as the right half in the cup subdiagram. Multiplying by $e_i$ we get
						$$
						\begin{pspicture}[shift=-2](-0.5,2)(3.5,-1.5)
						\psline[linecolor=gray](-0.5,0)(3.5,0)
						\psellipticarc(1.5,0)(0.5,0.5){180}{0}
						\psellipticarc(1.5,0)(1.5,1.2){180}{0}
						\psellipticarc(0.5,0)(0.5,0.5){0}{180}
						\psellipticarc(0.5,2)(0.5,0.5){180}{0}
						\psline(2,0)(2,2)
						\psline(3,0)(3,2)
						\end{pspicture}
						=
						\begin{pspicture}(0,0)(4,-0.5)
						\psellipticarc(1,0)(0.5,0.5){180}{0}
						\psellipticarc(3,0)(0.5,0.5){180}{0}
						\end{pspicture}
						$$
						which is exactly what we want.
					\item {\bf $s\leq n, t>n$:} We must restrict to the case that the orientation at $t$ is {$\up$}. So the last two terms in $K$ and $K_{new}$ disappear. In terms decorated tangles we get the multiplication
						$$
						\begin{pspicture}[shift=-2](-0.5,2.5)(3,-2)
						\psline[linecolor=gray](0,0)(3,0)
						\psellipticarc(2,0)(0.5,0.5){180}{0}
						\psline(0.5,0)(0.5,-2)
						\psellipticarc(1,0)(0.5,0.5){0}{180}
						\psellipticarc(1,2)(0.5,0.5){180}{0}
						\psline(2.5,0)(2.5,2)
						\end{pspicture}
						=
						\begin{pspicture}[shift=-1](0,0.5)(3.5,-2)
						\psellipticarc(1,0)(0.5,0.5){180}{0}
						\psline(2.5,0)(2.5,-2)
						\end{pspicture}
						$$						
					\item {\bf $s,t>n$:} This time, the weights must have {$\up$}'s at $s$ and $t$. In $K$ only the first term survives and $K_{new}$ is just $0$. This agrees with the cup subdiagrams since the new diagram is not a cup subdiagram of any cup diagram: otherwise we had a {$+$} at $s$, which is not possible, since $s$ is greater than $n$. In terms of tangles we get
						$$
						\begin{pspicture}[shift=-2](0,2.5)(2.5,-2)
						\psline[linecolor=gray](0,0)(2,0)
						\psline(0.5,0)(0.5,-2)
						\psline(1.5,0)(1.5,-2)
						\psellipticarc(1,0)(0.5,0.5){0}{180}
						\psellipticarc(1,2)(0.5,0.5){180}{0}
						\rput(0.5,-2){\footnotesize\boldmath$\times$}
						\rput(1.5,-2){\footnotesize\boldmath$\times$}
						\end{pspicture}
						\stackrel{\eqref{actionrel}}{=}
						0
						$$
					\end{enumerate}
				\end{enum}
		In all the cases the diagrammatical action of the tangles coincides with the action of the Hecke algebra. This proves that the diagram in the theorem commutes. The second assertion is clear. The faithfulness is proved in the next section.
\end{proof}

\subsection{The cellular algebra $\operatorname{\widehat{TL}}(\W)$}
In \cite[Thm. 4.2]{Gre98}, Green gives an explicit description of a basis $\mathcal{B}$ of $\operatorname{TL}(\W)$ in terms of decorated tangles as follows:
\begin{thm}
The basis $\mathcal{B}$ consists of $(n,n)$-tangles $T$ with at most one decoration on each loop or edge satisfying one of the following two conditions:
\begin{enum}
	\item $T$ contains one loop which is dotted, and no other loops or decorations; also, there is at least one non-propagating edge in the diagram.
	\item $T$ contains no loops and the total number of decorations is even.
\end{enum}
\end{thm}
Because of Relation $(iii)$ in \eqref{actionrel} the elements of the first type are $0$ while the other survive in $\operatorname{\widehat{TL}}(\W)$ and clearly form a basis. We denote the elements in this basis by $b_i$, $i\in I$,  with the unit element $b_0$.

\begin{proof}[Proof of faithfulness for $n$ odd]
Denote the decorated tangles in $\operatorname{dC}(n)$ by $d_j$, $j\in J$. Let $b=\sum_{i\in I}\alpha_i b_i\in\operatorname{\widehat{TL}}(\W)$ such that $b$ acts trivially, i.e. $d.b=0$ for all $d=\sum_{j\in J}\beta_jd_j\in\Nmod_t$. We have to deduce $b=0$ or equivalently $\alpha_i=0$ for all $i$.

Let $C=C_0$ be the unique element in $\operatorname{dC}(n)$ containing no cups. Note that
\begin{eqnarray*}
A(C)&:=&\{i\in I\mid \exists j\:: d_j.b_i=C\}=\{0\}.
\end{eqnarray*}
Then the coefficient of $C$ in $C.b$ equals $\alpha_0$, hence $\alpha_0=0$. Assume now that $\alpha_i=0$ for all $i$ satisfying the condition that there exists some $j$ such that $d_j.b_i$ has at most $m-1$ cups. Call the set of such $i$'s $A_{<m}$.
Now choose $C\in\operatorname{dC}(n)$ with $m$ cups (arbitrarily dotted or undotted). Consider the set
 \begin{eqnarray*}
A(C)&:=&\{i\in I\mid \exists j\:: d_j.b_i=C, i\notin A_{<m}\}.
\end{eqnarray*}
Choose $i_0\in A(C)$ such that $\alpha_{i_0}$ has maximal $q$-degree. Take $b_{i_0}$, cut in the middle and reflect horizontally to obtain an even decorated cup diagram $C_{i_0}$ after possibly inserting or removing a dot on the edge to make sure it is even (the edge always exists, since $n$ is odd).
Then we obtain, using that $\alpha_i=0$ for $i\in  A_{<m}$, the equalities
\begin{eqnarray*}
&0=C_{i_0}.b=\sum\limits_{i\in I}\alpha_i C_{i_0}.b_i=\sum\limits_{i\in A_{C}}\alpha_i C_{i_0}.b_i+\sum\limits_{i\not\in A_{<m}\cup  A_{C}}\alpha_i C_{i_0}.b_i&
\end{eqnarray*}
Since $C$ does not occur by definition in the second summand, the coefficient, say $\gamma_C$, of $C$ in the first summand must be zero. We want to show that $\alpha_{i_0}=0$. We claim that $C_{i_0}.b_{i_0}=\gamma C$ for some $\gamma\in\L$. Indeed, the cups contained in $b_{i_0}$ are also contained in $C$, since there exists some $d_j$ such that $d_j.b_{i_0}=C$ and the cups  in $b_{i_0}$ will not get changed. On the other hand, the number of cups in $b_{i_0}$ equals $m$, the number of cups in $C$, since otherwise $i_0\in A_{<m}$. Hence the claim follows. Moreover, $\gamma=(q+q^{-1})^m$, since putting $b_{i_0}$ on top of $C_{i_0}$ creates exactly $m$ internal circles, each of them with an even number of dots. Moreover, putting $b_{i}$ on top of $C_{i_0}$ for $i\in A_{C}$ creates either circles with an odd number of  dots or altogether less than $m$ circles. Hence $\gamma_{i_0}$ is a polynomial in $q$ of degree $m$ with leading term equal to $\alpha_{i_0}$ and so $\alpha_{i_0}=0$. Now we can proceed by induction on $m$ and the $q$-degree.
Hence the action is faithful.
\end{proof}

\begin{proof}[Proof of faithfulness for $n$ even]
Note that multiplying a generator of $I$ with some $e_i$ either annihilates it or keeps the maximal number of cups and caps. Similarly to Lemma \ref{lem:eventoeven} one shows that the defining property of $I$ is preserved. Hence $I$ is an ideal. With the same notation and arguments as in the case $n$ odd we obtain $\alpha_i=0$ for all $i\in A_{<\frac{n}{2}}$ and all $i$ such that $b_i$ has $\frac{n}{2}$ cups with an even number of them undotted. Note that any $b_i$ with $\frac{n}{2}$ cups and an odd number of them undotted kills any element in $\operatorname{dC}(n)$ by Lemma \ref{lem:eventoeven}. Hence, the action factors through $I$ and is faithful when passing to the quotient.
\end{proof}

\begin{rem}
The distinction between $n$ even and odd has a nice interpretation in terms of Springer fibers and nilpotent orbits, see \cite{ES}.
\end{rem}

Let $\Lambda=\{0,1,\ldots, n\}$ with the usual total ordering on integers, and for $\lambda\in\Lambda$ denote by $M(\lambda)$ the set of $(\lambda,n)$ even dotted tangles. Given $T\in M(\la)$ let $T^*$ denote the tangle obtained by reflecting $T$ horizontally obtaining an $(n, \lambda)$-diagram. Note that acting with $X\in\operatorname{\widehat{TL}}(\W)$ on $T$ produces a linear multiple of a $(\lambda',n)$-tangle where $\lambda'\leq\lambda$.
		
\begin{thm}
\label{thm:cellular}
The algebra  $\operatorname{\widehat{TL}}(\W)$ is a cellular algebra in the sense of Graham and
Lehrer \cite{GL} with cell datum $(\Lambda, M, C, *)$ where
\begin{enumerate}[a.)]
\item $M(\la)$ denotes the set of even dotted $(\lambda,n)$-tangles
\item $C$ is defined by setting
$C^\la_{\alpha,\beta}$ the diagram obtained by putting $\alpha$ on top of $\beta^*$ and concatenate
\item $*:\operatorname{\widehat{TL}}(\W)\rightarrow \operatorname{\widehat{TL}}(\W)$ is the anti-automorphism obtained by reflecting the tangle diagrams horizontally.
\end{enumerate}
\end{thm}

\begin{proof}
Recall that being a cell datum $(\Lambda, M, C, *)$ of the associative algebra $H=\operatorname{\widehat{TL}}(\W)$ means
\begin{itemize}
\item
$\La$ is a partially ordered set and $M(\la)$
is a finite set
for each $\la \in \La$;
\item
$C:\dot\bigcup_{\la \in \La} M(\la)
\times M(\la) \rightarrow H,
(\alpha,\beta) \mapsto C^\la_{\alpha,\beta}$ is an injective map
whose image is a basis for $H$;
\item
the map $*:H \rightarrow H$
is an algebra anti-automorphism such that
$(C^\la_{\alpha,\beta})^* = C_{\beta,\alpha}^\la$
for all $\la \in \La$ and $\alpha, \beta \in M(\la)$;
\item
if $\mu \in \La$ and $\gamma, \delta \in M(\la)$
then for any $x \in H$ we have that
\begin{eqnarray}
\label{cell4}
x C_{\gamma,\delta}^\mu &\equiv&\sum_{\gamma' \in M(\mu)} r_x(\gamma',\gamma) C_{\gamma',\delta}^\mu
\pmod{H(< \mu)}
\end{eqnarray}
where the scalar $r_x(\gamma',\gamma)$ is independent of $\delta$ and
$H(< \mu)$ denotes the subspace of $H$
generated by $\{C_{\gamma'',\delta''}^\nu\mid\nu < \mu,
\gamma'',\delta'' \in M(\nu)\}$.
\end{itemize}
The first, third and fourth assertions are clear. For the second note that $C^\la_{\alpha,\beta}$ is obviously nonzero and contains an even number of decorations. Namely if $\alpha$ contains $c$ undotted cups and $d$ dotted edges,  then the number of dotted cups is $n-\lambda-c$ and the total number of dots $d+n-\lambda-c$ which is congruent $n-\la$ modulo $2$. In particular, $C^\la_{\alpha,\beta}$  has an even number of dots and this property is not affected by removing two dots from a common edge. The result is then always a basis vector of $H$. On the other hand, given a basis vector of $H$ for $n$ odd, we can cut it horizontally into a pair of an even decorated cup diagram and an even decorated cap diagram. The result is unique except that there is a choice where to move the possible existing unique dot on an edge. With the above calculation we see however that there is a unique choice where to move the dot such that the decorated cup and cap diagram are both even. Hence $C$ is a bijection onto the set of basis vectors. In case $n$ is even, the argument is the same since we factor out precisely those tangles which cannot be cut to get an even decorated cup diagram and an even decorated cap diagram.
\end{proof}

\begin{rem}
{\rm Via \eqref{cell4}, the span of the set $M(\la)$ can be equipped with a $\operatorname{\widehat{TL}}(\W)$-module structure as follows: for $\alpha\in M(\la)$ choose $\beta$ such that $C^\la_{\alpha,\beta}$ is defined. Then the action of $x\in H$ is given by taking the product $xC^\la_{\alpha,\beta}$ from \eqref{cell4} modulo $H(< \la)$. The result is a linear combination of elements $C^\la_{\alpha',\beta}$, independent of $\beta$. Hence we can forget $\beta$ and obtain our desired module structure.
The modules $M(\la)$ are the {\it cell modules} in the sense of \cite{GL}. Theorem \ref{cupcap} implies that $\mathcal{N}_t$ has a filtration of cell modules $M(\la)$ for $0\leq \la\leq n$ with $\la\equiv n\;\op{mod} 2$ and each such $M(\la)$ appearing exactly once. The picture in Example \ref{cupcap} nicely shows this filtration in case $n=4$.
Using the method of little groups, \cite{Serre}, one shows that the above cell modules are indeed irreducible $\W$-modules. Our Theorem \ref{cupcap} can then be used to categorify these irreducible $\W$-modules using subquotient categories of  $\O^\p_0(\mathfrak{so}_{2n})$.}
\end{rem}

\begin{example}
{\rm Let $n=3$ then $\operatorname{\widehat{TL}}(\W)$ is $10$-dimensional with the map $C$ illustrated in the following diagrams:
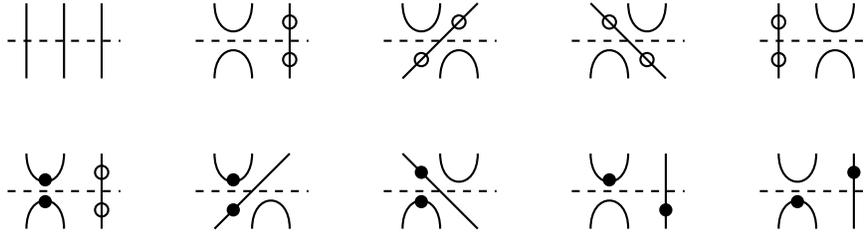
\begin{equation*}
\begin{tikzpicture}[thick]
\draw (.5,0) to +(0,-1);
\draw (1,0) to +(0,-1);
\draw (1.5,0) to +(0,-1);
\draw[dashed] (.25,-.5) to +(1.5,0);

\begin{scope}[xshift=2.5cm]
\draw (.5,0) .. controls +(0,-.5) and +(0,-.5) .. +(.5,0);
\draw (.5,-1) .. controls +(0,.5) and +(0,.5) .. +(.5,0);
\draw (1.5,0) to +(0,-1);
\draw (1.5,-.25) circle(2.5pt) (1.5,-.75) circle(2.5pt);
\draw[dashed] (.25,-.5) to +(1.5,0);
\end{scope}

\begin{scope}[xshift=5cm]
\draw (.5,0) .. controls +(0,-.5) and +(0,-.5) .. +(.5,0);
\draw (1,-1) .. controls +(0,.5) and +(0,.5) .. +(.5,0);
\draw (1.5,0) to +(-1,-1);
\draw (1.25,-.25) circle(2.5pt) (.75,-.75) circle(2.5pt);
\draw[dashed] (.25,-.5) to +(1.5,0);
\end{scope}

\begin{scope}[xshift=7.5cm]
\draw (1,0) .. controls +(0,-.5) and +(0,-.5) .. +(.5,0);
\draw (.5,-1) .. controls +(0,.5) and +(0,.5) .. +(.5,0);
\draw (.5,0) to +(1,-1);
\draw (.75,-.25) circle(2.5pt) (1.25,-.75) circle(2.5pt);
\draw[dashed] (.25,-.5) to +(1.5,0);
\end{scope}

\begin{scope}[xshift=10cm]
\draw (.5,0) to +(0,-1);
\draw (1,0) .. controls +(0,-.5) and +(0,-.5) .. +(.5,0);
\draw (1,-1) .. controls +(0,.5) and +(0,.5) .. +(.5,0);
\draw (.5,-.25) circle(2.5pt) (.5,-.75) circle(2.5pt);
\draw[dashed] (.25,-.5) to +(1.5,0);
\end{scope}

\begin{scope}[yshift=-2cm]

\draw (.5,0) .. controls +(0,-.5) and +(0,-.5) .. +(.5,0);
\fill (.75,-.35) circle(2.5pt);
\draw (.5,-1) .. controls +(0,.5) and +(0,.5) .. +(.5,0);
\fill (.75,-.64) circle(2.5pt);
\draw (1.5,0) to +(0,-1);
\draw (1.5,-.25) circle(2.5pt) (1.5,-.75) circle(2.5pt);
\draw[dashed] (.25,-.5) to +(1.5,0);

\begin{scope}[xshift=2.5cm]
\draw (.5,0) .. controls +(0,-.5) and +(0,-.5) .. +(.5,0);
\fill (.75,-.35) circle(2.5pt);
\draw (1,-1) .. controls +(0,.5) and +(0,.5) .. +(.5,0);
\draw (1.5,0) to +(-1,-1);
\fill (.75,-.75) circle(2.5pt);
\draw[dashed] (.25,-.5) to +(1.5,0);
\end{scope}

\begin{scope}[xshift=5cm]
\draw (1,0) .. controls +(0,-.5) and +(0,-.5) .. +(.5,0);
\draw (.5,-1) .. controls +(0,.5) and +(0,.5) .. +(.5,0);
\fill (.75,-.64) circle(2.5pt);
\draw (.5,0) to +(1,-1);
\fill (.75,-.25) circle(2.5pt);
\draw[dashed] (.25,-.5) to +(1.5,0);
\end{scope}

\begin{scope}[xshift=7.5cm]
\draw (.5,0) .. controls +(0,-.5) and +(0,-.5) .. +(.5,0);
\fill (.75,-.35) circle(2.5pt);
\draw (.5,-1) .. controls +(0,.5) and +(0,.5) .. +(.5,0);
\draw (1.5,0) to +(0,-1);
\fill (1.5,-.75) circle(2.5pt);
\draw[dashed] (.25,-.5) to +(1.5,0);
\end{scope}

\begin{scope}[xshift=10cm]
\draw (.5,0) .. controls +(0,-.5) and +(0,-.5) .. +(.5,0);
\draw (.5,-1) .. controls +(0,.5) and +(0,.5) .. +(.5,0);
\fill (.75,-.64) circle(2.5pt);
\draw (1.5,0) to +(0,-1);
\fill (1.5,-.25) circle(2.5pt);
\draw[dashed] (.25,-.5) to +(1.5,0);
\end{scope}

\end{scope}
\end{tikzpicture}
\end{equation*}
The unfilled circles are $\bullet$'s which are added to obtain a pair of even decorated cup/cap diagrams. There are two cell modules, the $1$-dimensional module $M(3)$ and the $3$-dimensional module $M(1)$. Over $\mathbb{C}(q)$ the algebra is semisimple (as a quotient of the semisimple algebra $\mathbb{C}(q)\otimes_{\mathbb{Z}[q,q^{-1}]} \H$, see e.g. \cite{Mathas}). The cell modules are then irreducible, since $10=1^2+3^2$.
}
\end{example}

\section{Algebra structure and Koszul gradings}
Similar to \cite{BS08:III}, the vector space $E^\p$ can be equipped with a diagrammatically defined associative algebra structure, compatible with the grading given by the $\op{deg}$-function. This algebra provides a graded version of Braden's algebra in type $D$ from \cite{Braden}, see \cite{ES} for details.

In \cite{BGS}, a graded version $\O^\p_0(\mathfrak{so}_{2n})^\mathbb{Z}$ of  $\O^\p_0$ was defined and it was shown that $E^\p$ can be equipped with a Koszul grading. This grading and graded versions of translation functors were studied in detail in \cite{Stgrad}. The results then imply that  Proposition~\ref{dottedcirclediags} can be refined
\begin{thm}
The graded piece of degree $i$ in
$\operatorname{Hom}_{\O^\p_0(\mathfrak{so}_{2n})^\mathbb{Z}}(P(w),P(w'))$ is of dimension equal to the number of oriented circle diagrams $\overline{w}v\underline{w'}$ with $\op{deg}=i$.
\end{thm}
Hence computing the graded dimension of $E^\p$ translates via our calculus into counting the number of clockwise cups and caps. Similarly, Theorem~\ref{cupcap} has a graded version following (\cite[(1.2)]{Stgrad}, \cite[Theorem 2.4]{Str05}).

These results together with \cite{CV} predict the existence of a graded version of the Brauer algebras $\op{Br}_r(\delta)$ for $\delta\in\mathbb{Z}$ similar to the analogous (but technically much easier) construction in \cite{BS08:V} for the walled Brauer algebra. Details about these aspects as well as connections of our diagram combinatorics with the combinatorics for orthysymplectic Lie superalgebra from \cite{GS} and super Schur-Weyl duality, \cite{CLW}, will appear in \cite{ES}.
\bibliographystyle{alpha}
\bibliography{bibliography}
\end{document}